\documentclass[12pt, reqno]{amsart}
\usepackage[utf8]{inputenc}
\usepackage[english]{babel} %Lingua

\usepackage{amssymb}
\usepackage{amsthm}  %Teoremi
\usepackage{amsmath} %\text
\usepackage{enumerate} %Liste con numeri romani
\usepackage{mathrsfs}
\usepackage{upgreek} % tau del troncamento
\usepackage{fix-cm} % per cambiare sizefont quando voglio, comando \begin{sizepar}{ grandezza voluta} { + 6} o \size{}{}
\usepackage{adjustbox}
\usepackage{verbatim} % COMMENTS
\usepackage{stmaryrd} % LLBRACKET
\usepackage{dirtytalk}
\usepackage{pdflscape}
\usepackage[mathscr]{eucal}
\usepackage{microtype}

\usepackage{tikz-cd}

\usepackage{aliascnt}

\usepackage{graphicx}

\calclayout

\usepackage[inner = 25mm, outer = 25mm, top = 25mm, bottom = 32mm, a4paper]{geometry}
\makeatletter
\def\l@subsection{\@tocline{2}{0pt}{2.5pc}{5pc}{}}
\makeatother

\DeclareRobustCommand{\SkipTocEntry}[5]{}

\usepackage{xcolor} %colorix
\usepackage[colorlinks=true, linkcolor=blue, urlcolor=red, citecolor=brown]{hyperref} %cosÃ¬ decido i colori dei link hyperref
\addto\extrasenglish{

  }
  
\let\oldfootnotemark\footnotemark
\let\oldfootnotetext\footnotetext
\let\oldfootnote\footnote
\renewcommand\footnote[1]{\addtocounter{footnote}{1}\hypertarget{fnbackref.\arabic{footnote}}{}\addtocounter{footnote}{-1}\oldfootnote{#1\fnbackref}}
\renewcommand\footnotemark{\addtocounter{footnote}{1}\hypertarget{fnbackref.\arabic{footnote}}{}\addtocounter{footnote}{-1}\oldfootnotemark}
\renewcommand\footnotetext[1]{\oldfootnotetext{#1\fnbackref}}
\newcommand{\fnbackref}{\hyperlink{fnbackref.\arabic{footnote}}{\footnotesize$\uparrow$}}

\theoremstyle{definition}
\newtheorem{dfn}{Definition}[subsection]

\newaliascnt{assumption}{dfn}

\aliascntresetthe{assumption}

\theoremstyle{remark}

\newaliascnt{rmk}{dfn}
\newtheorem{rmk}[rmk]{Remark}
\aliascntresetthe{rmk}

\newaliascnt{ex}{dfn}
\newtheorem{ex}[ex]{Example}
\aliascntresetthe{ex}

\theoremstyle{plain}
\newaliascnt{thm}{dfn}
\newtheorem{thm}[thm]{Theorem}
\aliascntresetthe{thm}

\newaliascnt{prop}{dfn}
\newtheorem{prop}[prop]{Proposition}
\aliascntresetthe{prop}

\newaliascnt{prop2}{dfn}

\aliascntresetthe{prop2}

\newaliascnt{lem}{dfn}
\newtheorem{lem}[lem]{Lemma}
\aliascntresetthe{lem}

\newaliascnt{cor}{dfn}
\newtheorem{cor}[cor]{Corollary}
\aliascntresetthe{cor}

\DeclareMathOperator{\Hom}{Hom}

\DeclareMathOperator{\im}{im}

\newcommand{\derived}{\mathrm{D}}
\newcommand{\bounded}{\mathrm{b}}

\newcommand{\qc}{\mathrm{qc}}

\newcommand{\mb}[1]{\mathbb{#1}}
\newcommand{\ca}[1]{\mathcal{#1}}
\newcommand{\mcr}[1]{\mathscr{#1}}

\newenvironment{sizepar}[2]
    {\fontsize{#1}{#2}\selectfont}
    {\par}

\makeatletter
\newcommand{\nodeequation}[1]{%
  \let\label\ltx@label
  \refstepcounter{equation}%
  #1
  \quad
  (\theequation)%
}
\makeatother

\newcommand{\ind}[1]{L \mathrm{Ind}_{#1}}
\newcommand{\res}[1]{\mathrm{Res}_{#1}}
\newcommand{\oL}[1]{\stackrel{L}{\otimes}_{\ca{#1}}}
\newcommand{\ob}[1]{\, \overline{\otimes}_{\ca{#1}} \, }
\newcommand{\oo}[1]{\otimes_{\ca{#1}}}
\newcommand{\rmodule}[2]{\left(\begin{array}{cc} #1 & #2 \end{array}\right)}
\newcommand{\lmodule}[2]{\left(\begin{array}{c} #1 \\ #2 \end{array}\right)}
\newcommand{\bimodule}[4]{\left(\begin{array}{cc} #1 & #2 \\ #3 & #4 \end{array}\right)}
\newcommand{\barr}[1]{\overline{\ca{#1}}}
\newcommand{\barmod}[1]{\overline{{\bf Mod}}\text{-}\ca{#1}}
\newcommand{\bimod}[2]{\ca{#1}\text{-}{\bf Mod}\text{-}\ca{#2}}
\newcommand{\stmod}[1]{{\bf Mod}\text{-}\ca{#1}}
\newcommand{\barbimod}[2]{\ca{#1}\text{-}\overline{{\bf Mod}}\text{-}\ca{#2}}
\newcommand{\dual}[2]{#1^{\barr{#2}}}
\newcommand{\add}[1]{#1^{\mathrm{add}}}

\begin{document}

\title{{\bf On the composition of two spherical twists}}

\author{Federico Barbacovi}
\address{Department of Mathematics, University College London}
\email{federico.barbacovi.18@ucl.ac.uk}

\begin{abstract}
  \begin{comment}
  Spherical functors provide a formal way to package autoequivalences of enhanced triangulated categories.
  Moreover, E. Segal proved that any autoequivalence of an enhanced triangulated category can be realized as a spherical twist.

  When exhibiting an autoequivalence as a spherical twist, one has various choices for the source category of the spherical functor.
  We prove that, in the DG setting, given two spherical twists there is a natural way to produce a new spherical functor whose twist is the composition of the spherical twists we started with, and whose source category semiorthogonally decomposes into the source categories for the two spherical functors we began with.
  \end{comment}

   E. Segal proved that any autoequivalence of an enhanced triangulated category can be realised as a spherical twist.
   However, when exhibiting an autoequivalence as a spherical twist one has various choices for the source category of the spherical functor.
   We describe a construction that realises the composition of two spherical twists as the twist around a single spherical functor whose source category semiorthogonally decomposes into the source categories for the spherical functors we started with.

   We give a description of the cotwist for this spherical functor and prove, in the special case when our starting twists are around spherical objects, that the cotwist is the Serre functor (up to a shift).
   We finish with an explicit treatment for the case of $\mb{P}$-objects.
\end{abstract}

\maketitle

\tableofcontents

%% INTRODUCTION
%%%
%%%
%%% INTRODUCTION
%%%

\section{Introduction}
In \cite{Seg-autoeq-spherical-twist}, Segal proved that every autoequivalence of a triangulated category can be realised as a spherical twist around a spherical functor.
The construction of the spherical functor is explicit but, as pointed out in {\it ibidem}, the source category for this functor is not optimal, in the sense that it could be larger than needed.
\begin{comment}
The author gives a way to deform this category in the case of autoequivalences with a section, {\it i.e.} an autoequivalence $F : \ca{C} \rightarrow \ca{C}$ together with a natural transformation $\sigma : \text{id} \rightarrow F$.
\end{comment}

An example of spherical twist is given by the spherical twist around a spherical object as introduced in \cite{Seidel-Thomas01}.
Let $X$ be a smooth projective variety of dimension $d$ over a field $k$, and consider an object $E \in \derived^{\bounded} (\text{Coh}(X))$ such that $\text{RHom}_{X}(E,E) = k \oplus k[-d]$ and $E \otimes \omega_{X} \simeq E$.
Then, one can define an autoequivalence $T_{E} \colon \derived^{\bounded}(\text{Coh}(X)) \rightarrow \derived^{\bounded}(\text{Coh}(X))$ such that, for every $F \in \derived^{\bounded} (\text{Coh}(X))$, the object $T_E(F)$ fits in a distinguished triangle
\[
	\text{RHom}_{X}(E, F ) \otimes E \rightarrow F \rightarrow T_{E}(F).
\]
In the setting of spherical functors, the functor $T_{E}$ can be realised as the spherical twist around the spherical functor $F \colon \derived^{\bounded}(\text{pt}) \rightarrow \derived^{\bounded}(\text{Coh}(X))$ that sends $k$ to $E$.
It is clear that in this case $\derived^{\bounded}(\text{pt})$ is the nicest category we can look for.

Assume now that we have two spherical objects $E_1,E_2 \in \derived^{\bounded}(\text{Coh}(X))$; we know that the autoequivalence $\Phi = T_{E_2} \circ T_{E_1}$ can be realised as a spherical twist around a spherical functor, therefore it is natural to ask ourselves what is this functor and what is the best source category for it.
As we are dealing with two objects, one might expect the morphisms between them to play a role in the description of the functor $\Phi$.
This expectation is indeed correct, and to represent $\Phi$ as a spherical twist around a spherical functor we will use the dg-algebra
\[
	R = k \, \mathrm{id}_{E_2} \oplus \text{RHom}_{X}(E_2, E_1) \oplus k \, \mathrm{id}_{E_1},
\]
where the multiplication comes from the algebra structure of $\mathrm{RHom}_{X}(E_2 \oplus E_1, E_2 \oplus E_1)$.

Let us now consider the general setting of spherical functors as developed in \cite{Anno-Log-17}: let $\ca{A},\ca{B},\ca{C}$ be three small dg-categories over a field $k$, and $M \in \derived(\ca{A}\text{-}\ca{C})$, $N \in \derived(\ca{B}\text{-}\ca{C})$ be two spherical bimodules, {\it i.e.}, bimodules whose associated functors are spherical.

We aim to find a spherical functor whose twist gives the composition of the twists associated to $M$ and $N$.
The following is our main result.

\begin{thm}[\autoref{twist-twist=twist}]
	Let $\ca{A}$, $\ca{B}$ and $\ca{C}$ be three small dg-categories over a field $k$, and $M \in \derived(\ca{A}\text{-}\ca{C})$, $N \in \derived(\ca{B}\text{-}\ca{C})$ be two  spherical bimodules.
	Then, the bimodule $P := \rmodule{M}{N}^t \in \derived((\ca{B} \sqcup_{\varphi} \ca{A})\text{-}\ca{C})$, $\varphi = \textup{RHom}_{\ca{C}}(N,M)$, with structure morphism
	\[
		\mathrm{RHom}_{\ca{C}}(N,M) \oL{B} N \xrightarrow{\mathrm{tr}} M
	\]
	is spherical, the twist around it is given by the composition $t_N \circ t_M$, and the cotwist is described by the matrix
	\[
		C_P = \left(
		\begin{array}{cc}
			C_M & 0 \\
			\mathrm{RHom}_{\ca{C}}(M,N)[-1] & C_N
		\end{array}
		\right)
		\in \derived((\ca{B} \sqcup_{\varphi} \ca{A})\text{-}(\ca{B} \sqcup_{\varphi} \ca{A}))
	\]
	with non-zero structure morphisms
	\[
		\begin{aligned}
			& \, \mathrm{RHom}_{\ca{C}}(M,N)[-1] \oL{A} \mathrm{RHom}_{\ca{C}}(N,M) \xrightarrow{\sigma \circ \mathrm{cmps}} C_N\\
			& \, \mathrm{RHom}_{\ca{C}}(N,M) \oL{B}  \mathrm{RHom}_{\ca{C}}(M,N)[-1] \xrightarrow{\sigma \circ \mathrm{cmps}} C_M,
		\end{aligned}
	\]
	where $\mathrm{cmps}$ denotes the composition of morphisms.
\end{thm}

For an explanation of the notation we refer to \autoref{subsect:gluing}, \autoref{subsect:modules-on-glued} and \autoref{dfn-sigma}.

\begin{rmk}
	In \cite[Section 7.2]{AL-P-n-func}, to provide an example of a non-split $\mb{P}^n$-functor, the authors perfomed the above construction in the special case where $M = 0$.
\end{rmk}

As a corollary to the previous theorem, we get the following

\begin{cor}[\autoref{cor:commutativity-relation}]
	With the same setting as in the above theorem, we have $t_N \circ t_M \simeq t_{t_N(M)} \circ t_N$.
\end{cor}

Finally, we deal with two explicit examples: spherical objects and $\mb{P}$-objects.
In the case of spherical objects we can prove the following

\begin{thm}[\autoref{thm:cotwist-serre-duality}]
	Let $\ca{C}$ be a small dg-category over a field $k$ such that $\derived(\ca{C})^c$ has a Serre functor, and let $E_1, \dots, E_n \in \derived(\ca{C})$ be $d$-spherical objects.
	Then, the cotwist around the spherical functor realising the composition $t_{E_n} \circ \dots \circ t_{E_1}$ gives a Serre functor up to a shift.
\end{thm}

We refer the reader to \autoref{subsect:spherical-objs} for a more detailed statement of the above theorem.

\addtocontents{toc}{\SkipTocEntry}
\subsection*{Acknowledgments}

I would like to thank my advisor Ed Segal for many helpful conversations.
I would also like to thank Rina Anno and Timothy Logvinenko for carefully reading the paper and providing useful feedback.
Furthermore, I would like to thank Daniel Halpern-Leistner for asking how the main theorem of this paper related to \cite[Theorem 4.14]{Halpern-Shipman16}, Tobias Dyckerhoff for an interesting discussion that resulted in the proof of \autoref{cor:commutativity-relation}, and the referees whose comments and suggestions incredibly improved the paper and brought to a description of the cotwist in full generality.

This project has received funding from the European Research Council (ERC) under the European Union Horizon 2020 research and innovation programme (grant agreement No.725010).

%% GENERALITIES
%
%
%
%
%
\section{Generalities}

In this section we briefly recall the notions and results we need.
Moreover, we prove a few technical results regarding semiorthogonal decompositions of triangulated categories and gluing of dg-categories.
All functors throughout will be covariant and we work over a fixed field $k$.

%
%
%
% SOD TRIANGULATED CATEGORIES
%
%
%

\subsection{Semiorthogonal decompositions for triangulated categories}

Let $\mcr{T}$ be a triangulated category and $\mcr{F}$, $\mcr{G}$ be two full triangulated subcategories.
We give the following

\begin{dfn}[\cite{BonKap89}]
	\label{def-sod}
	We say that $\mcr{F}$ and $\mcr{G}$ give a SemiOrthogonal Decomposition of $\mcr{T}$ if for any $F \in \mcr{F}$, $G \in \mcr{G}$ we have $\Hom_{\mcr{T}}(G,F) = 0$, and for any $T \in \mcr{T}$ there exists a distinguished triangle
	\begin{equation}
		\label{triangle-sod}
			T_{\mcr{G}} \rightarrow T \rightarrow T_{\mcr{F}},
	\end{equation}
	where $T_{\mcr{F}} \in \mcr{F}$, $T_{\mcr{G}} \in \mcr{G}$.
	In this case, we write $\mcr{T} = \langle \mcr{F}, \mcr{G} \rangle$.
\end{dfn}

\begin{rmk}
	The existence of a SOD implies the existence of a left adjoint for the inclusion $i_{\mcr{F}} \colon \mcr{F} \hookrightarrow \mcr{T}$ and a right adjoint for the inclusion $i_{\mcr{G}} \colon \mcr{G} \hookrightarrow \mcr{T}$.
	We will call these adjoints the projection functors.
\end{rmk}

\begin{dfn}
	A {\it left gluing functor} for a SOD $\mcr{T} = \langle \mcr{F}, \mcr{G} \rangle$ is a functor $\phi \colon \mcr{F} \rightarrow \mcr{G}$ such that there exist natural isomorphisms
	\[
		\begin{array}{lcr}
			\Hom_{\mcr{T}}(F[-1], G) \xrightarrow[\psi_{F,G}]{\simeq} \Hom_{\mcr{G}} (\phi(F), G) & & \forall \, F \in \mcr{F}, \, G \in \mcr{G}.
		\end{array}
	\]
	Similarly, a {\it right gluing functor} for a SOD $\mcr{T} = \langle \mcr{F}, \mcr{G} \rangle$ is a functor $\phi \colon \mcr{G} \rightarrow \mcr{F}$ such that there exist natural isomorphisms
	\[
		\begin{array}{lcr}
			\Hom_{\mcr{T}}(F, G[1]) \xrightarrow[\psi_{F,G}]{\simeq} \Hom_{\mcr{F}} (F, \phi(G)) & & \forall \,  F \in \mcr{F}, \, G \in \mcr{G}.
		\end{array}
	\]
	We will suppress the dependence of $\psi$ on $F$ and $G$ from now on.
\end{dfn}

\begin{comment}
Notice that, when we have a SOD with a left gluing functor, objects $T \in \mcr{T}$ are in bijection with triples $(M,N, \mu)$ where $M \in \mcr{F}$, $N \in \mcr{G}$, and $\mu : \phi(M) \rightarrow N$ is a morphism in $\mcr{G}$.
The bijection sends $T \in \mcr{T}$ to $(T_{\mcr{F}}, T_{\mcr{G}}, \mu_T)$, where $T_{\mcr{F}} \in \mcr{F}$ and $T_{\mcr{G}} \in \mcr{G}$ are the image of the left, right respectively, adjoint to the inclusion of $\mcr{F}$ and $\mcr{G}$, and $\mu_T : \phi(T_{\mcr{F}}) \rightarrow T_{\mcr{G}}$ is the map corresponding to the structure morphism\footnote{This is the content of \cite[Lemma 2.5]{KL15}.}
\[
	\begin{tikzcd}
		T_{\mcr{F}} \ar[r] & T_{\mcr{G}}[1].
	\end{tikzcd}
\]
\end{comment}

We will call a triangulated category $\mcr{T}$ closed under arbitrary direct sums {\it cocomplete};
a {\it continuous} functor $F \colon \mcr{T} \rightarrow \mcr{T}'$ between cocomplete triangulated categories is a functor which commutes with arbitrary direct sums.

\begin{dfn}
	\label{dfn:cpt-obj}
	Let $\mcr{T}$ be a cocomplete triangulated category.
	An object $T \in \mcr{T}$ is called \emph{compact} if the canonical morphism
	\[
		\displaystyle{\Hom_{\mcr{T}}(T, \bigoplus_{i} Q_i)} \rightarrow \displaystyle{\bigoplus_i \Hom_{\mcr{T}}(T, Q_i)}
	\]
	is an isomorphism for any family $Q_i$.
	The full subcategory of compact objects is denoted by $\mcr{T}^{c}$.
\end{dfn}

\begin{dfn}
	A \emph{cocomplete subcategory} $\mcr{S}$ of a cocomplete category $\mcr{T}$ is a subcategory which is cocomplete and such that the inclusion functor $i_{\mcr{S}} \colon \mcr{S} \rightarrow \mcr{T}$ is continuous.
\end{dfn}

\begin{rmk}
	There exist subcategories of cocomplete categories which are cocomplete but which are not cocomplete subcategories according to the above definition.
	Examples of such subcategories can be constructed as follows.
	Consider $f \colon X \rightarrow Y$ a map of Noetherian schemes over a field $k$ of characteristic zero such that\footnote{Our functors are implicitly derived.} $f_{\ast}$ does not preserve compactness and $f_{\ast} \ca{O}_X \simeq \ca{O}_Y$, e.g. a resolution of rational singularities.
	Then, $f^{\times}$ is a fully faithful functor because for any $E \in \derived_{\qc}(Y)$ we have\footnote{The functor $R\ca{H}om_{D_{qc}(Y)}(-,-)$ is the internal hom functor for the tensor structure on $D_{qc}(Y)$. In general it is quite hard to compute, but the hypothesis $f_{\ast}\ca{O}_{X} \simeq \ca{O}_Y$ makes it easy in our case.}
	\[
		f_{\ast} f^{\times} E \simeq \mathrm{R}\ca{H}om_{\derived_{\qc}(Y)}(f_{\ast}\ca{O}_{X}, E) \simeq E,
	\]
	see \cite[Remark 6.1.1]{Neeman-Grothendieck-duality} for the first isomorphism.
	In particular, $f^{\times}\derived_{\qc}(Y) \subset \derived_{\qc}(X)$ is a subcategory which is abstractly cocomplete but the inclusion functor cannot preserve infinite direct sums because its left adjoint is $f_{\ast}$, which does not preserve perfectness.
\end{rmk}

Our aim is now to investigate the relation between a SOD and the subcategory of compact objects.

\begin{lem}
	\label{compact-obj-dec}
	Let $\mcr{T}$ be a cocomplete triangulated category and assume we have a semiorthogonal decomposition $\mcr{T} = \langle \mcr{F}, \mcr{G} \rangle$ in terms of cocomplete subcategories.
	Furthermore, assume that there exists a left gluing functor $\phi : \mcr{F} \rightarrow \mcr{G}$.
	If $\phi$ preserves compact objects, then the inclusion and projection functors preserve compactness, and we have $\mcr{T}^{c} = \langle \mcr{F}^{c}, \mcr{G}^{c} \rangle$.
\end{lem}

\begin{rmk}
	\label{rmk:proj-commutes-sums}
	Notice that we know that there exists a distinguished triangle
	\[
		i_{\mcr{G}} i_{\mcr{G}}^R \rightarrow \text{id} \rightarrow i_{\mcr{F}} i_{\mcr{F}}^L.
	\]
	In particular, as $i_{\mcr{F}}^L$ is continuous being a left adjoint, and as $i_{\mcr{F}}$ and $i_{\mcr{G}}$ are fully faithful, both projection functors commute with arbitrary direct sums.
\end{rmk}

We introduce the following notation
\[
	\Hom^{\bullet}_{\mcr{T}}(-,-) := \bigoplus_{n \in \mb{Z}} \Hom_{\mcr{T}}(-,-[n])[-n].
\]

\begin{proof}
	For any $T \in \mcr{T}$ consider distinguished triangle \eqref{triangle-sod}.
	We aim to prove that $T$ is compact if and only if $T_{\mcr{F}}$ is compact in $\mcr{F}$ and $T_{\mcr{G}}$ is compact in $\mcr{G}$.
	Assume $T$ is compact, then as $\mcr{G}^{\perp} = \{ T \in \mcr{T} : \Hom_{\mcr{T}}(G, T) = 0 \quad \forall \, G \in \mcr{G} \} = \mcr{F}$ and $\mcr{F}$ is a cocomplete subcategory, we readily obtain that $T_{\mcr{F}}$ is compact.
	Let now $G_{i} \in \mcr{G}$ be a subset of objects, $i \in \ca{I}$.
	If  we apply the functor $\Hom^{\bullet}_{\mcr{T}}(-, \bigoplus_{i \in \ca{I}} G_i)$ to the distinguished triangle \eqref{triangle-sod} associated to $T$, we get the distinguished triangle
	\[
		\displaystyle{\Hom^{\bullet}_{\mcr{T}}(T_{\mcr{F}}, \bigoplus_{i \in \ca{I}} G_i)} \rightarrow \displaystyle{\Hom^{\bullet}_{\mcr{T}}(T, \bigoplus_{i \in \ca{I}} G_i)} \rightarrow \displaystyle{\Hom^{\bullet}_{\mcr{T}}(T_{\mcr{G}}, \bigoplus_{i \in \ca{I}} G_i)}.
	\]
	As $T$ is compact, we have the isomorphism $\Hom^{\bullet}_{\mcr{T}}(T  , \bigoplus_{i \in \ca{I}} G_i) \simeq \bigoplus_{i \in \ca{I}} \Hom^{\bullet}_{\mcr{T}}(T, G_i)$.
	Moreover, as $T_{\mcr{F}}$ is compact, $\mcr{G}$ is a cocomplete subcategory, and $\phi$ preserves compactness, we have
	\[
		\begin{aligned}
			\Hom^{\bullet}_{\mcr{T}}( T_{\mcr{F}} , \bigoplus_{i \in \ca{I}} G_i) &\simeq \Hom^{\bullet}_{\mcr{G}}(\phi(T_{\mcr{F}}[1]), \bigoplus_{i \in \ca{I}} G_{i})\\
			&\simeq \bigoplus_{i \in \ca{I}} \Hom^{\bullet}_{\mcr{G}}(\phi(T_{\mcr{F}}[1]), G_i)\\
			&\simeq \bigoplus_{i \in \ca{I}} \Hom^{\bullet}_{\mcr{T}}(T_{\mcr{F}}, G_i).
		\end{aligned}
	\]
	Therefore, we also have $\Hom^{\bullet}_{\mcr{T}}(T_{\mcr{G}}, \bigoplus_{i \in \ca{I}} G_{i}) \simeq \bigoplus_{i \in \ca{I}} \Hom^{\bullet}_{\mcr{T}}(T_{\mcr{G}}, G_{i})$, \emph{i.e.}, $T_{\mcr{G}}$ is compact.

	Conversely, let us assume that $T_{\mcr{F}}$ and $T_{\mcr{G}}$ are compact.
	Let $T_i$, $i \in \ca{I}$, be objects in $\mcr{T}$, and let $T_{\mcr{F},i}$, $T_{\mcr{G},i}$ be their projections in $\mcr{F}$ and $\mcr{G}$ respectively.
	By \autoref{rmk:proj-commutes-sums}, we have the distinguished triangle
	\begin{equation}
		\label{triangle-direct-sum}
			\displaystyle{\bigoplus_{i \in \ca{I}}} T_{\mcr{G},i} \rightarrow \displaystyle{\bigoplus_{i \in \ca{I}} T_i} \rightarrow \displaystyle{\bigoplus_{i \in \ca{I}} T_{\mcr{F},i}.}
	\end{equation}
	Applying the functor $\Hom^{\bullet}_{\mcr{T}}( T_{\mcr{G}} , -)$ to this triangle and using that $\mcr{G}^{\perp} = \mcr{F}$, we get $\Hom^{\bullet}_{\mcr{T}}(T_{\mcr{G}}, \bigoplus_{i \in \ca{I}} T_i) \simeq \bigoplus_{i \in \ca{I}} \Hom^{\bullet}_{\mcr{T}} ( T_{\mcr{G}}, T_i)$.
	Then, applying the functor $\Hom^{\bullet}_{\mcr{T}}(T_{\mcr{F}}, -)$ to \eqref{triangle-direct-sum} we get
	\[
		\displaystyle{\Hom^{\bullet}_{\mcr{T}}(T_{\mcr{F}}, \bigoplus_{i \in \ca{I}} T_{\mcr{G},i}) \rightarrow \Hom^{\bullet}_{\mcr{T}}(T_{\mcr{F}}, \bigoplus_{i \in \ca{I}} T_i) \rightarrow \Hom^{\bullet}_{\mcr{T}}(T_{\mcr{F}}, \bigoplus_{i \in \ca{I}} T_{\mcr{F},i})}.
	\]

	As $T_{\mcr{F}}$ is compact and $\mcr{F}$ is a cocomplete subcategory, we get $\Hom^{\bullet}_{\mcr{T}} (T_{\mcr{F}},  \bigoplus_{i \in \ca{I}} T_{\mcr{F},i}) \simeq \bigoplus_{i \in \ca{I}} \Hom^{\bullet}_{\mcr{T}}( T_{\mcr{F}}, T_{\mcr{F},i})$.
	Moreover, using that $T_{\mcr{F}}$ is compact, $\mcr{G}$ is a cocomplete subcategory, and $\phi$ preserves compactness, we get the isomorphism $\Hom^{\bullet}_{\mcr{T}}( T_{\mcr{F}},  \bigoplus_{i \in \ca{I}} T_{\mcr{G},i}) \simeq \bigoplus_{i \in \ca{I}} \Hom^{\bullet}_{\mcr{T}}( T_{\mcr{F}}, T_{\mcr{G},i})$.
	Therefore, $\Hom^{\bullet}_{\mcr{T}}(T_{\mcr{F}}, \bigoplus_{i \in \ca{I}} T_i) \simeq \bigoplus_{i \in \ca{I}} \Hom^{\bullet}_{\mcr{T}} ( T_{\mcr{F}}, T_i)$.
	This isomorphism and the one for $T_{\mcr{G}}$ imply $\Hom^{\bullet}_{\mcr{T}}(T, \bigoplus_{i \in \ca{I}} T_i) \simeq \bigoplus_{i \in \ca{I}} \Hom^{\bullet}_{\mcr{T}} ( T, T_i)$, {\it i.e.} $T$ is compact.
\end{proof}

Similarly, one can prove

\begin{lem}
	\label{compact-obj-dec-2}
	Let $\mcr{T}$ be a cocomplete triangulated category and assume we have a semiorthogonal decomposition $\mcr{T} = \langle \mcr{F}, \mcr{G} \rangle$ in terms of cocomplete subcategories.
	Furthermore, assume that there exists a right gluing functor $\phi : \mcr{G} \rightarrow \mcr{F}$.
	If $\phi$ is continuous, then the inclusion and projection functors preserve compactness, and $\mcr{T}^{c} = \langle \mcr{F}^{c}, \mcr{G}^{c} \rangle$.
\end{lem}

\begin{rmk}
	Notice that in general it is not true that a SOD $\mcr{T} = \langle \mcr{F}, \mcr{G} \rangle$ induces a SOD for the subcategory of compact objects.
	For a counterexample, see \autoref{ex:inclusion-doesnt-preserve-cpt}.
\end{rmk}

%
%
%
% dg CATEGORIES
%
%
%

\subsection{(Modules over) dg-categories}

For what follows, our main reference is \cite{Anno-Log-17}.

As we said in the introduction, our aim is to describe a construction to present the composition of two spherical twists around spherical functors as a single twist.
Even though spherical twists were first introduced as autoequivalences of triangulated categories, see \cite{Seidel-Thomas01}, it turns out that the correct setup for such objects is that of dg-categories.

Let us fix a field $k$.
The category ${\bf Mod}\text{-}k$ will be the category of dg-modules over $k$, {\it i.e.} graded $k$-modules $V = \bigoplus_{n \in \mb{Z}} V^n$ equipped with a $k$-linear endomorphism $d_V$ of degree $1$, called the {\it differential}, such that $d_V^2 = 0$.
For any two dg $k$-modules $(V,d_V)$, $(W,d_W)$, we define the hom space between them as
\[
	\Hom_{{\bf Mod}\text{-}k}((V,d_V), (W,d_W)) = \displaystyle{\bigoplus_{n \geq 0} \Hom^n((V,d_V), (W,d_W))},
\]
where $f \in \Hom^n((V,d_V), (W,d_W))$ is a homomorphism of $k$ vector spaces $f : V \rightarrow W$ such that $f(V^p) \subset W^{p+n}$.
We endow this graded $k$-module with the differential given by
\[
	d \left( \{f_n\}\right) = \{ d_W \circ f_n - (-1)^n f_n \circ d_V \}.
\]
The tensor product of $(V,d_V)$ and $(W,d_W)$ is defined as $(V \otimes_k W)^n = \oplus_{i+j = n} V^i \otimes_k W^j$ with differential $d_V \otimes \text{id} + \text{id} \otimes d_W$.

A dg-category $\ca{A}$ is a category enriched over\footnote{Notice that we do not require a dg-category to be additive.} ${\bf Mod}\text{-}k$, {\it i.e.}, for any $a,a' \in \ca{A}$ the hom space $\Hom_{\ca{A}}(a,a')$ is an object in ${\bf Mod}\text{-}k$, and the composition maps
\[
	\Hom_{\ca{A}}(a',a'') \otimes_k \Hom_{\ca{A}}(a,a') \rightarrow \Hom_{\ca{A}}(a,a'')
\]
are maps of dg $k$-modules for any $a, a', a'' \in \ca{A}$.

A dg functor $F \colon \ca{A} \rightarrow \ca{B}$ between two dg-categories is a functor such that for any $a, a' \in \ca{A}$ the map $\Hom_{\ca{A}}(a,a') \xrightarrow{F} \Hom_{\ca{B}}(F(a), F(a'))$ is a morphism in ${\bf Mod}\text{-}k$.

A right dg-module over $\ca{A}$ is a dg-functor $\ca{A}^{\mathrm{opp}} \rightarrow {\bf Mod}\text{-}k$.
An example of right dg-module is given for any $a \in \ca{A}$ by $h^a(-) = \Hom_{\ca{A}}(-, a)$.
The dg functor $a \mapsto h^{a}$ is fully faithful and it is called the {\it Yoneda embedding}.
For simplicity, we will denote $\Hom_{\stmod{A}}(-,-) = \Hom_{\ca{A}}(-,-)$.
We say that a module $S \in \stmod{A}$ is {\it acyclic} if for every $a \in \ca{A}$ the complex $S(a)$ is acyclic.
The full subcategory of acyclic modules is denoted by $Acycl(\ca{A})$.
We define the derived category of $\ca{A}$ as the Verdier quotient
\[
	\derived(\ca{A}) = H^0(\stmod{A}) \left/ H^0(Acycl(\ca{A})) \right.,
\]
where, given a dg-category $\ca{B}$, $H^0(\ca{B})$ is the category with the same objects as $\ca{B}$ and with morphisms
\[
	\begin{array}{lcr}
		\Hom_{H^0(\ca{B})}(b,b') = H^0 ( \Hom_{\ca{B}}(b,b') ) & & \forall \, b,b' \in H^0(\ca{B}).
	\end{array}
\]

\begin{dfn}
	A module $P \in \stmod{A}$ is called {\it h-projective} if for any $S \in Acycl(\ca{A})$ we have $\mathrm{Hom}_{H^0(\stmod{A})}(P,S) = 0$.
	The full subcategory of h-projective modules is denoted by $\ca{P}(\ca{A})$.
\end{dfn}

\begin{dfn}
	A module $M \in \stmod{A}$ is called {\it perfect} if $M$ is a compact object in $\derived(\ca{A})$.
\end{dfn}

Given two dg-categories $\ca{A}$ and $\ca{B}$, an $\ca{A}\text{-}\ca{B}$ bimodule is a functor $\ca{A} \otimes_{k} \ca{B}^{\mathrm{opp}} \rightarrow {\bf Mod}\text{-}k$, where $\ca{A} \otimes_{k} \ca{B}^{\mathrm{opp}}$ is the tensor product of dg-categories.
We will denote the category of $\ca{A}\text{-}\ca{B}$ bimodules $\bimod{A}{B}$ and its derived category $\derived(\ca{A}\text{-}\ca{B})$.
For a bimodule $M \in \bimod{A}{B}$, we denote the image of $(a,b) \in \ca{A} \otimes_k \ca{B}^{\mathrm{opp}}$ as ${}_a M_b$.

\begin{ex}
	For a dg-category $\ca{A}$, the diagonal bimodule is given by ${}_{a} \ca{A}_{b} = \Hom_{\ca{A}}(b, a)$.
\end{ex}

\begin{dfn}
	A bimodule $M \in \bimod{A}{B}$ is called $\ca{A}$-perfect if for any $b \in \ca{B}$ the module $M_{b}$ is $\ca{A}$-perfect.
	Similarly, we define $\ca{B}$-perfectness and $\ca{A}$ and $\ca{B}$ h-projectivity.
\end{dfn}

\begin{rmk}
	\label{rmk:Keller-h-proj}
	As we are working over a field $k$, for any dg-category $\ca{A}$ the diagonal bimodule has a functorial h-projective resolution $\overline{\ca{A}}$, the {\it bar resolution}, \cite{Kel94}.
	Moreover, any h-projective $\ca{A}\text{-}{\ca{B}}$ bimodule is automatically $\ca{A}$ and $\ca{B}$ h-projective.

	Using the bar resolution one can construct functorial h-projective resolutions using e.g. \cite[Proposition 2.5]{Anno-Log-17}.
\end{rmk}

One can define a tensor product functor $- \oo{B} - : \bimod{A}{B} \otimes_k \bimod{B}{C} \rightarrow \bimod{A}{C}$.
Moreover, using h-projective resolutions, this functor can be derived.
In particular, to any $M \in \bimod{A}{B}$ we can associate a functor $f_M \colon \derived(\ca{A}) \rightarrow \derived(\ca{B})$ by the formula $f_M(S) = S \oL{A} M$.

Given $M \in \stmod{A}$, the $\ca{A}$-dual of $M$ is defined as the $\ca{A}^{\mathrm{opp}}$-module that assigns to any $a \in \ca{A}$ the complex $\Hom_{\ca{A}}(M, {}_a\ca{A})$ and is denoted by $M^{\ca{A}}$.
This construction gives rise to the {\it dualising functor} $\left( \stmod{A} \right)^{\mathrm{opp}} \rightarrow \stmod{A}^{\mathrm{opp}}$, $M \mapsto M^{\ca{A}}$.
We can derive the functor $(-)^{\ca{A}}$ and obtain the {\it derived dualising functor}, which will be denoted by $(-)^{\tilde{\ca{A}}}$.
Notice that these constructions could be performed with bimodules, resulting in bimodules rather than modules.

Given $M \in \bimod{A}{B}$, we have adjunctions
\begin{equation}
\label{underived-adjunctions}
	\begin{array}{lcr}
		- \otimes_{\ca{A}} M \dashv \Hom_{\ca{B}}(M, -) & \mathrm{and} &  M \otimes_{\ca{B}} - \dashv \Hom_{\ca{A}^{\mathrm{opp}}}(M, -)
	\end{array}
\end{equation}
which can be derived to adjunctions
\begin{equation}
\label{derived-adjunctions}
	\begin{array}{lcr}
		- \oL{A} M \dashv \text{RHom}_{\ca{B}}(M, -) & \mathrm{and} &  M \oL{B} - \dashv \text{RHom}_{\ca{A}^{\mathrm{opp}}}(M, -).
	\end{array}
\end{equation}

We will call the unit and counit of adjunctions \eqref{derived-adjunctions} evaluated at the diagional bimodules the {\it derived action map} and {\it derived trace map}, respectively, and given any $N \in \derived(\ca{B})$ we will denote
\[
	\mathrm{tr} : \mathrm{RHom}_{\ca{B}}(M,N) \stackrel{L}{\otimes}_{\ca{A}} M \rightarrow N
\]
the derived trace map of $M$ evaluated at $N$.

In the following, we will need our functors to have left and right adjoints which are still tensor functors.
The theorem below shows that it is enough to restrict our attention to bimodules which are perfect on both sides.
For a proof, see e.g. \cite[Theorem 4.1]{Anno-Logvinenko-Bar-Categories}.

\begin{thm}
	\label{thm:adjoints}
	Let $M \in \bimod{A}{B}$ and consider the induced functor $f_M \colon \derived(\ca{A}) \rightarrow \derived(\ca{B})$.
	Then, the following are equivalent:
	\begin{enumerate}
		\item The right adjoint of $f_M$ exists and is continuous (resp. the left adjoint exists).
		\item The right (resp. left) adjoint functor of $f_M$ is given by $- \oL{B} M^{\widetilde{\ca{B}}}$ (resp. $- \oL{B} M^{\widetilde{\ca{A}}}$).
		\item $M$ is $\ca{B}$- (resp. $\ca{A}$-) perfect.
		\item $f_M$ (resp. $M \oL{A} -$) preserves compactness.
	\end{enumerate}
\end{thm}

\begin{comment}
We conclude with the following

\begin{dfn}
	Given $M \in \stmod{C}$, $N \in \stmod{C}$, we define the {\it evaluation map} as
	\[
		N \otimes_{\ca{C}} M^{\ca{C}} \xrightarrow{\mathrm{ev}} \Hom_{\ca{C}}(M,N)
	\]
	sending for any $c \in \ca{C}$
	\[
		\begin{array}{lcr}
			n \otimes \varphi \to \left( m \to n\varphi(m) \right) & & n \in N_{c}, \, \varphi \in \Hom_{\ca{C}}( M, \ca{C}_{c}).
		\end{array}
	\]
\end{dfn}
\end{comment}

\subsubsection{Functors between derived categories}

In the previous paragraphs we introduced functors between derived categories of dg-categories.
All of them were constructed using a bimodule and deriving a functor that existed at the level of the category of modules.
We now want to explain two other constructions which will turn out to be useful later on and that let us induce functors between derived categories starting with dg-functors.

\begin{dfn}
\label{dfn:induction-restriction-functors}
	Given a dg-functor $F \colon \ca{A} \rightarrow \ca{B}$, we define 
	\[
		\begin{array}{ll}
			\text{Res}_{F} \colon \stmod{B} \rightarrow \stmod{A}, & M \mapsto \left( a \mapsto M_{F(a)} \right),\\
			\text{Ind}_F \colon \stmod{A} \rightarrow \stmod{B}, & M \mapsto M \otimes_{\ca{A}} {}_{F} \ca{B},
		\end{array}
	\]
	where
	\[
		 {}_{a}({}_{F} \ca{B})_{b} = \Hom_{\ca{B}}(b, F(a)).
	\]
\end{dfn}

The tensor-hom adjunction \eqref{underived-adjunctions} tells us that we have $\text{Ind}_F \dashv \text{Res}_F$.
Moreover, an explicit computation shows $\text{Ind}_F ( h^a) \simeq h^{F(a)}$ for $a \in \ca{A}$.

The functor $\text{Res}_F$ clearly maps acyclic modules to acyclic modules, and therefore descends to a functor between derived categories.
By adjunction, we get that $\text{Ind}_F$ preserves h-projective modules, and that descends to the derived category using h-projective resolutions.

We will write $L\text{Ind}_{F}$ for the derived functor of $\text{Ind}_F$ and $M_F = \text{Res}_F(M)$.
We have

\begin{prop}[$\text{\cite[Proposition 3.9]{KL15}}$]
	\label{prop:KL-ind-res}
	The functor $L\textup{Ind}_{F}$ is left adjoint to the functor $\textup{Res}_F$ and both functors commute with arbitrary direct sums.
	Moreover, we have $L \textup{Ind}_F(h^a) \simeq h^{F(a)}$.

	If $F$ induces a fully faithful functor $H^0(F) \colon H^0(\ca{A}) \rightarrow H^0(\ca{B})$, then $L\textup{Ind}_F$ is fully faithful.
	Finally, if $H^0(F)$ is an equivalence, so is $L\mathrm{Ind}_F$.
\end{prop}

Let us now assume that we have a dg-functor $G \colon \ca{B} \rightarrow \ca{A}$ together with natural isomorphisms of chain complexes
\begin{equation}
	\label{adj-F-G}
	\begin{array}{lcr}
		\Hom_{\ca{B}} (F(a), b) \simeq \Hom_{\ca{A}} (a, G(b)) & & \forall \; a \in \ca{A}, \, b \in \ca{B}.
	\end{array}
\end{equation}
Then, we have

\begin{prop}
	\label{res-res-adjunction}
	We have an adjunction $\textup{Res}_F \dashv \textup{Res}_{G}$ at the level of derived categories.
	In particular, we have $L\textup{Ind}_G \simeq \textup{Res}_F$.
\end{prop}
\begin{proof}
	Notice that the right adjoint to $\text{Res}_{F}$ exists by Brown representability, \cite{Neeman-Brown-Rep}.
	Let us call it $H$.
	Then, for any $b \in \ca{B}$, $N \in \derived(\ca{A})$, we have
	\[
	\begin{aligned}
		\Hom_{\derived(\ca{B})} ( h^{b}, H(N) ) \simeq & \, \Hom_{\derived(\ca{A})} ( \text{Res}_{F}(h^{b}), N) \\
		\simeq & \, \Hom_{\derived(\ca{A})} ( h^{G(b)}, N) \\
		\simeq & \, \Hom_{\derived(\ca{B})}(h^b, N_G),
	\end{aligned}
	\]
	where we used \eqref{adj-F-G} to pass from the first to the second line.
	Therefore, we get an isomorphism $H(N) \simeq N_G$ from the Yoneda lemma.
	The second statement follows from the uniqueness of the adjoint.
\end{proof}

\subsection{Bar categories}

In this section we introduce the formalism of bar categories as developed in \cite{Anno-Logvinenko-Bar-Categories}.
We will recall only the basic notions we will need, and we refer to {\it ibidem} and references therein for a thorough treatment.

Let us fix $\ca{A}$ a small dg-category.
Then, to $\ca{A}$ we can associate its {\it bar complex} $\barr{A} \in \bimod{A}{A}$ as defined in \cite[Definition 2.24]{Anno-Logvinenko-Bar-Categories}.
The bar complex is an h-projective bimodule which can be equipped with the structure of a unital coalgebra in the monoidal category $(\bimod{A}{A}, \otimes_{\ca{A}}, \ca{A})$, see \cite[Proposition 2.33]{Anno-Logvinenko-Bar-Categories}.
We will denote the counit and the comultiplication as 
\begin{equation}
	\label{projection-comultiplication}
	\begin{array}{lcr}
		\barr{A} \xrightarrow{\tau} \ca{A} & \mathrm{and} & \barr{A} \xrightarrow{\Delta} \barr{A} \otimes_{\ca{A}} \barr{A},
	\end{array}
\end{equation}
and we will sometime refer to $\tau$ as the projection morphism.

It is well known that the counit $\tau$ is a quasi-isomorphism, and therefore $\barr{A}$ gives an h-projective resolution of the diagonal bimodule, see \cite{Kel94}.

The following is \cite[Definition 3.2]{Anno-Logvinenko-Bar-Categories}.

\begin{dfn}
	The {\it bar category of modules} $\barmod{A}$ is defined as follows:
	\begin{itemize}
		\item The objects are given by dg-modules over $\ca{A}$.
		\item For any $E, F \in \stmod{A}$ set
		\[
			\mathrm{H}\overline{\mathrm{om}}_{\ca{A}}(E,F) := \mathrm{Hom}_{\barmod{A}}(E,F) := \mathrm{Hom}_{\ca{A}}(E \otimes_{\ca{A}} \barr{A}, F). 
		\]
		\item For any $E \in \stmod{A}$ we set $\mathrm{id}_E \in \overline{\mathrm{Hom}}_{\ca{A}}(E,E)$ to be
		\[
			E \otimes_{\ca{A}} \barr{A} \xrightarrow{\mathrm{id} \otimes \tau} E \otimes_{\ca{A}} \ca{A} \xrightarrow{\simeq} E
		\]
		\item For any $E,F,G \in \stmod{A}$ the composition of $E \otimes_{\ca{A}} \barr{A} \xrightarrow{\alpha} F$ and $F \otimes_{\ca{A}} \barr{A} \xrightarrow{\beta} G$ is the element given by
		\[
			E \otimes_{\ca{A}} \barr{A} \xrightarrow{\mathrm{id} \otimes \Delta} E \otimes_{\ca{A}} \barr{A} \otimes_{\ca{A}} \barr{A} \xrightarrow{\alpha \otimes \mathrm{id}} F \otimes_{\ca{A}} \barr{A} \xrightarrow{\beta} G.
		\]
	\end{itemize}
\end{dfn}

Similarly one can define the bar category of bimodules $\barbimod{A}{B}$ for two small dg-categories $\ca{A}$, $\ca{B}$.

The following is \cite[Corollary 3.6]{Anno-Logvinenko-Bar-Categories}.

\begin{thm}
	\label{thm:enhancement}
	There is a canonical category isomorphism
	\[
		\Theta : \derived(\ca{A}) \xrightarrow{\simeq} H^0 \left( \barmod{A} \right)
	\]
	giving $\barmod{A}$ the structure of a dg-enhancement of $\derived(\ca{A})$.
\end{thm}

Let us now take three small dg-categories $\ca{A}$, $\ca{B}$, $\ca{C}$.
Then, one can define dg-functors
\[
	\renewcommand{\arraystretch}{1.5}
	\begin{array}{c}
		\ob{\ca{B}} : \barbimod{A}{B} \otimes_k \barbimod{B}{C} \rightarrow \barbimod{A}{C}\\
		\mathrm{H}\overline{\mathrm{om}}_{\ca{B}}(-,-) : \barbimod{A}{B} \otimes_k \left( \barbimod{C}{B} \right)^{\mathrm{opp}} \rightarrow \barbimod{A}{C}
	\end{array}
\]
as per \cite[Definition 3.9, 3.10]{Anno-Logvinenko-Bar-Categories} by setting
\[
	\begin{array}{lcr}
		M \ob{\ca{B}} N = M \otimes_{\ca{B}} \barr{B} \otimes_{\ca{B}} N & \mathrm{and} & 
		\mathrm{H}\overline{\mathrm{om}}_{\ca{B}}(M,N') = \mathrm{Hom}_{\ca{B}}(M \otimes_{\ca{B}} \barr{B}, N')
	\end{array}
\]
for $M \in \barbimod{A}{B}$, $N \in \barbimod{B}{C}$, and $N' \in \barbimod{C}{B}$.

For a fixed $M \in \barbimod{A}{B}$ the functors $(- \ob{A} M, \mathrm{H} \overline{\mathrm{om}}_{\ca{B}}(M, -))$ form an adjoint pair $\barmod{A} \leftrightarrow \barmod{B}$, and similarly for $(M \ob{B} - , \mathrm{H} \overline{\mathrm{om}}_{\ca{A^{\mathrm{opp}}}}(M, -))$, see \cite[Proposition 3.14]{Anno-Logvinenko-Bar-Categories}.
We now give an explicit description of the adjunction units evaluated at the diagonal bimodules $\ca{A}$ and $\ca{B}$, respectively.
We first define the following maps of $\barbimod{A}{B}$ bimodules
\begin{equation}
	\label{A-map}
	\barr{A} \oo{A} \barr{A} \oo{A} M \oo{B} \barr{B} \xrightarrow{\tau \oo{} \tau \oo{} \mathrm{id} \oo{} \tau} M
\end{equation}
\begin{equation}
	\label{B-map}
	\barr{A} \oo{A} M \oo{B} \barr{B} \oo{B} \barr{B} \xrightarrow{\tau \oo{} \mathrm{id} \oo{} \tau \oo{} \tau} M
\end{equation}
Then, the adjunction units evaluated at the diagonal bimodules are given by 
\begin{equation}
	\label{act-a}
	\barr{A} \oo{A} \barr{A} \xrightarrow{\mathrm{mult}} \mathrm{Hom}_{\ca{B}}(M \oo{B} \barr{B}, \barr{A} \oo{A} \barr{A} \oo{A} M \oo{B} \barr{B} ) \xrightarrow{\eqref{A-map} \circ (-)} \mathrm{H} \overline{\mathrm{om}}_{\ca{B}}(M,M)
\end{equation}
\begin{equation}
	\label{act-b}
	\barr{B} \oo{B} \barr{B} \xrightarrow{\mathrm{mult}} \mathrm{Hom}_{\ca{A}^{\mathrm{opp}}}(\barr{A} \oo{A} M, \barr{A} \oo{A} M \oo{B} \barr{B} \oo{B} \barr{B}) \xrightarrow{\eqref{B-map} \circ (-)} \mathrm{H} \overline{\mathrm{om}}_{\ca{A}^{\mathrm{opp}}}(M,M).
\end{equation}

We define the {\it dualising functors} $\barbimod{A}{B} \rightarrow \left( \barbimod{B}{A} \right)^{\mathrm{opp}}$ as
\[
	\begin{array}{lcr}
		\dual{(-)}{A} := \mathrm{H} \overline{\mathrm{om}}_{\ca{A}^{\mathrm{opp}}}(-, \ca{A})
		& \mathrm{and} &
		\dual{(-)}{B} := \mathrm{H} \overline{\mathrm{om}}_{\ca{B}}(-, \ca{B}),
	\end{array}
\]
see \cite[Definition 3.31]{Anno-Logvinenko-Bar-Categories}.
As showed in \cite[Definition 3.35]{Anno-Logvinenko-Bar-Categories}, one can construct natural transformations
\begin{equation}
	\label{nat-transf-bar-dual-1}
	\dual{M}{A} \ob{A} (-) \xrightarrow{\eta_{\ca{A}}} \mathrm{H} \overline{\mathrm{om}}_{\ca{A}^{\mathrm{opp}}}(M,-)
\end{equation}
\begin{equation}
	\label{nat-transf-bar-dual-2}
	(-) \ob{\ca{B}} \dual{M}{B} \xrightarrow{\eta_{\ca{B}}} \mathrm{H} \overline{\mathrm{om}}_{\ca{B}}(M,-)
\end{equation}
such that \eqref{nat-transf-bar-dual-1} (resp. \eqref{nat-transf-bar-dual-2}) is an homotopy equivalence if and only if $M$ is $\ca{A}$- (resp. $\ca{B}$-) perfect, see \cite[Lemma 3.36]{Anno-Logvinenko-Bar-Categories}.

\subsection{Adjunctions}

It is easy to see that the functors $\dual{(-)}{A}$ and $\dual{(-)}{B}$ are self adjoint, see e.g. \cite[§ 3.5]{Anno-Logvinenko-Bar-Categories}.
Hence, we get natural transformations $\mathrm{id} \rightarrow (-)^{\barr{A}\,\barr{A}}$ and $\mathrm{id} \rightarrow (-)^{\barr{B}\,\barr{B}}$ which are homotopy equivalences when evaluated on $\ca{A}$- and $\ca{B}$- perfect bimodules, respectively, see \cite[Lemma 3.32]{Anno-Logvinenko-Bar-Categories}.

Using \eqref{nat-transf-bar-dual-1} and \eqref{nat-transf-bar-dual-2} we get that for $M \in \barbimod{A}{B}$ an $\ca{A}$- and $\ca{B}$-perfect bimodule  the following form homotopy adjoint pair of functors
\begin{equation}
	\label{tensor-hom-adj}
	\begin{array}{lcr}
		\left( - \ob{A} M, - \ob{B} \dual{M}{B} \right),
		& & 
		\left( - \ob{B} \dual{M}{A}, - \ob{A} M \right).
	\end{array}
\end{equation}

We now perform the same choices as in \cite[§ 4.2]{Anno-Logvinenko-Bar-Categories}.
We refer the reader to {\it ibidem} for the complete explanation, and here we recall only the definitions we need.

\begin{dfn}
	Fix $M \in \barbimod{A}{B}$ an $\ca{A}$- and $\ca{B}$-perfect bimodule.
	\begin{itemize}
		\item The {\it homotopy trace maps}
		\begin{equation}
			\label{trace-maps}
			\begin{array}{lcr}
				M \ob{B} \dual{M}{A} \xrightarrow{\mathrm{tr}} \ca{A} & & \dual{M}{B} \ob{A} M \xrightarrow{\mathrm{tr}} \ca{B}
			\end{array}
		\end{equation}
		are the counit of the adjunctions \eqref{tensor-hom-adj} evaluated on the diagonal bimodules.
		\item We fix a homotopy inverse
		\begin{equation*}
			\mathrm{H} \overline{\mathrm{om}}_{\ca{A}^{\mathrm{opp}}}(M,M) \xrightarrow{\zeta_{\ca{A}}} \dual{M}{A} \ob{A} M
		\end{equation*}
		to \eqref{nat-transf-bar-dual-1}, and a homotopy inverse
		\begin{equation*}
			\mathrm{H} \overline{\mathrm{om}}_{\ca{B}}(M,M) \xrightarrow{\zeta_{\ca{B}}} M \ob{B} \dual{M}{B}
		\end{equation*}
		to \eqref{nat-transf-bar-dual-2}.
		Then, we define the {\it homotopy $\ca{A}$-action (resp. $\ca{B}$-action) map} as
		\begin{equation}
			\label{act-map-A}
				\mathrm{act}: \ca{A} \xrightarrow{\eqref{act-a}} \mathrm{H} \overline{\mathrm{om}}_{\ca{B}}(M,M) \xrightarrow{\zeta_{\ca{B}}} M \ob{B} \dual{M}{B}
		\end{equation}
		\begin{equation}
			\label{act-map-B}
				\mathrm{act}: \ca{B} \xrightarrow{\eqref{act-b}} \mathrm{H} \overline{\mathrm{om}}_{\ca{A}^{\mathrm{opp}}}(M,M) \xrightarrow{\zeta_{\ca{A}}} \dual{M}{A} \ob{A} M
		\end{equation}
	\end{itemize}
\end{dfn}

By \cite[Proposition 3.12]{Anno-Logvinenko-Bar-Categories} we know that the functor $- \ob{A} M$ (resp. $- \ob{B} \dual{M}{B}$, $- \ob{B} \dual{M}{A}$) enhances the functor $- \oL{A} M$ (resp. $- \oL{B} M^{\widetilde{\ca{B}}}, - \oL{B} M^{\widetilde{\ca{A}}}$) via the equivalence $\Theta$ of \autoref{thm:enhancement}.
The maps \eqref{trace-maps}, \eqref{act-map-A}, \eqref{act-map-B} are our fixed lifts at the homotopy level of the counits and units, respectively, of the adjunctions at the level of derived categories.
As the adjunctions \eqref{tensor-hom-adj} are only homotopy adjunctions, the standard relations between units and counits hold only up to homotopy.
However, these homotopies can be fixed together with the lifts \eqref{trace-maps}, \eqref{act-map-A}, \eqref{act-map-B}, see \cite[Proposition 4.6]{Anno-Logvinenko-Bar-Categories}.

\subsection{Gluing of dg-categories}
\label{subsect:gluing}

In this subsection we describe the notion of gluing of dg-categories.

There are two definitions of gluing in the literature: one for general dg-categories \cite{Tabuada-dgCAT}, and one that works best when the dg-categories are additive \cite{KL15}.

As we will be interested in the derived category of the gluing, choosing either model does not make a difference for us.
However, we have to balance two facts: that we want our results to hold for general dg-categories, and that working with the additive gluing is easier.
In \autoref{appendix} we will prove that it is not restrictive to assume that the dg-categories we work with are additive, thus solving this issue.

\subsubsection{The first gluing}

Let us take $\ca{A}$ and $\ca{B}$ two small dg-categories, and let $\varphi \in \bimod{A}{B}$ be a bimodule.
Following \cite{Tabuada-07}, we define the upper triangular dg-category associated to this datum as the dg-category $\ca{B} \sqcup_{\varphi} \ca{A}$ whose objects are
\[
	Obj(\ca{B} \sqcup_{\varphi} \ca{A}) = Obj(\ca{B}) \sqcup Obj(\ca{A}),
\]
and with complexes of morphisms
\[
	\Hom_{\ca{B} \sqcup_{\varphi} \ca{A}}(x,y) = \left\{
	\begin{array}{lc}
		\Hom_{\ca{B}}(x,y) & x,y \in \ca{B}\\
		\Hom_{\ca{A}}(x,y) & x,y \in \ca{A}\\
		{}_y \varphi_{x} & y \in \ca{A}, x \in \ca{B}\\
		0 & \text{otherwise}
	\end{array}
	\right..
\]
The grading, the differential, and the composition are defined in the obvious way.

In the following, we will denote 
\begin{equation}
	\label{eqn:inclusion-easy-gluing}
	\begin{array}{lcr}
		i_{\ca{B}} \colon \ca{B} \hookrightarrow \ca{B} \sqcup_{\varphi} \ca{A} & \mathrm{and} & i_{\ca{A}} \colon \ca{A} \hookrightarrow \ca{B} \sqcup_{\varphi} \ca{A}
	\end{array}
\end{equation}
the embedding functors.

\begin{ex}
	\label{ex:upper-triangular-dg-algebra}
	Let $A$ and $B$ two dg-algebras and $V$ an $A\text{-}B$ bimodule.
	Then, the upper triangular dg-algebra associated to this datum is the dg-algebra
	\[
		R = \bimodule{A}{V}{0}{B}
	\]
	with componentwise grading and differential, and composition law given by
	\[
		\bimodule{a}{v}{0}{b} \cdot \bimodule{a'}{v'}{0}{b'} = \bimodule{aa'}{av' + vb'}{0}{bb'}.
	\]

	If we think of $A$ and $B$ as dg-categories $\star_A$ and $\star_B$ with one object and endomorphism dg-algebra $A$ and $B$, respectively, then $V$ is a $\star_A\text{-}\star_B$ bimodule and ${\bf Mod}\text{-} \left( \star_B \sqcup_{V} \star_A \right) \simeq {\bf Mod}\text{-}(\star_R)$.
\end{ex}

\subsubsection{The additive gluing}
\label{subsubsection:gluing-dg}

Let us now take two small, additive dg-categories $\ca{A}$, $\ca{B}$ and an $\ca{A}\text{-}\ca{B}$ bimodule $\varphi$.
Following \cite[§ 4]{KL15} we define the gluing of $\ca{A}$ and $\ca{B}$ along $\varphi$, and we denote it by $\ca{B} \times_{\varphi} \ca{A}$, as follows: its objects are given by triples $(b, a, \mu)$ where
\[
	\begin{array}{ccc}
		b \in \ca{B}, & a \in \ca{A}, & \mu \in Z^{0} \left( {}_a\varphi_b\right),
	\end{array}
\]
and the morphisms are given by (here we set $r = (b, a, \mu)$, $r' = (b', a', \mu')$)
\[
	\Hom_{\ca{B} \times_{\varphi} \ca{A}} ( r, r' )
	= \Hom_{\ca{B}}( b, b') \oplus \Hom_{\ca{A}}( a, a') \oplus {}_{a'}\varphi_b [-1],
\]
with a suitable choice of differential and composition law described in \cite[§ 4.1]{KL15}.

We will denote
\begin{equation}
	\label{embedding-gluing}
	\begin{array}{llcrr}
		i_{\ca{B}} : \ca{B} \rightarrow \ca{R}, & b \mapsto (b, 0, 0) & \mathrm{and} & i_{\ca{A}} : \ca{A} \rightarrow \ca{R}, &  a \mapsto (0, a, 0)\\
	\end{array}
\end{equation}
the embedding functors, and
\begin{equation}
	\label{adjoint-embeddings-gluing}
	\begin{array}{llcrr}
		i_{\ca{B}}^{L} : \ca{R} \longrightarrow \ca{B}, & (b, a, \mu) \mapsto b & \mathrm{and} & i_{\ca{A}}^{R} : \ca{R} \longrightarrow \ca{A}, & (b, a, \mu) \mapsto a\\
	\end{array}
\end{equation}
their left and right adjoint, respectively.

\subsubsection{Relations between the two definitions}
\label{subsub:relations}

If $\ca{A}$ and $\ca{B}$ are small, additive dg-categories, we can perfom both the constructions we described in the previous subsections and obtain two new dg-categories: $\ca{B} \sqcup_{\varphi} \ca{A}$, $\ca{B} \times_{\varphi} \ca{A}$.

The relation between the two is as follows: the category $\ca{B} \sqcup_{\varphi} \ca{A}$ can be identified with the full subcategory of $\ca{B} \times_{\varphi[1]} \ca{A}$ of objects of the form $(b, 0, 0)$, $(0, a, 0)$.
Moreover, the fully faithful functor $F : \ca{B} \sqcup_{\varphi} \ca{A} \hookrightarrow \ca{B} \times_{\varphi[1]} \ca{A}$ induces an equivalence of derived categories $\derived(\ca{B} \sqcup_{\varphi} \ca{A}) \xrightarrow[\simeq]{\ind{F}} \derived(\ca{B} \times_{\varphi[1]} \ca{A})$.
Indeed, we obtain fully faithfulness from \autoref{prop:KL-ind-res}, and then we get essential surjectivity by continuity and the fact that the essential image of $\ind{F}$ contains the set of compact generators given by modules of the form $h^{(b,0,0)}$, $h^{(0,a,0)}$.

The fact that the two gluings have the same derived category is what really matters to us.
In \autoref{appendix} we will investigate the derived category of $\ca{B} \sqcup_{\varphi} \ca{A}$ in the case when $\ca{A}$ and $\ca{B}$ are not additive.

\subsection{Semiorthogonal decompositions for glued dg-categories}

Let $\ca{A}$ and $\ca{B}$ two small, additive dg-categories and $\varphi \in \bimod{A}{B}$.
We denote $\ca{R} := \ca{B} \times_{\varphi} \ca{A}$ the gluing defined in \autoref{subsubsection:gluing-dg} and 
\[
	\renewcommand*{\arraystretch}{1.2}
	\begin{array}{llcrr}
		I_1 = L\text{Ind}_{i_{\ca{B}}}, & I_2 = L\text{Ind}_{i_\ca{A}} & \mathrm{and} & I_1^R = \text{Res}_{i_{\ca{B}}}, & I_2^R = \text{Res}_{i_{\ca{A}}}
	\end{array}
\]
where $i_{\ca{A}}$ and $i_{\ca{B}}$ are defined in \eqref{embedding-gluing}.

By \autoref{prop:KL-ind-res} the functors $I_1$ and $I_2$ are fully faithful.
We use them to embed the categories $\derived(\ca{B})$ and $\derived(\ca{A})$ in $\derived(\ca{R})$.
We have

\begin{prop}[$\text{\cite[Proposition 4.6]{KL15}}$]
	\label{SOD-prop-KL}
	There exists a semiorthogonal decomposition
	\begin{equation}
		\label{SOD-KL}
		\derived(\ca{R}) = \langle I_1(\derived(\ca{B})), I_2(\derived(\ca{A})) \rangle
	\end{equation}
	with right gluing functor given by $- \stackrel{L}{\otimes}_{\ca{A}} \varphi: \derived(\ca{A}) \rightarrow \derived(\ca{B})$.
	Moreover, for any $F \in \derived(\ca{R})$ there is a distinguished triangle
	\begin{equation}
		\label{dt-sod-KL}
			I_1^R(F) \rightarrow I^{L}_1 F \rightarrow I^R_2(F) \oL{A} \varphi,
	\end{equation}
	where $I^{L}_{1}$ is the left adjoint of $I_1$.
\end{prop}

\begin{rmk}
	\label{rmk:to-cpt-objs}
	Notice that, as $- \stackrel{L}{\otimes}_{\ca{A}} \varphi$ is continuous, the hypotheses of \autoref{compact-obj-dec-2} are satisfied, and therefore we get $\derived(\ca{R})^c = \langle I_1(\derived(\ca{B})^c), I_2(\derived(\ca{A})^c) \rangle$.
\end{rmk}

Let us now consider the functor $I_3^R = \text{Res}_{i_{\ca{A}}^R}$.
Notice that the technical hypothesis of \autoref{res-res-adjunction} is satisfied, therefore we have an adjunction $I_2^R \dashv I_3^R$.

\begin{prop}
	\label{SOD-prop-mine}
	There exists a semiorthogonal decomposition
	\begin{equation}
		\label{SOD-mine}
		\derived(\ca{R}) = \langle I_3^R(\derived(\ca{A})), I_1(\derived(\ca{B})) \rangle
	\end{equation}
	with left gluing functor given by $- \stackrel{L}{\otimes}_{\ca{A}} \varphi[-1]: \derived(\ca{A}) \rightarrow \derived(\ca{B})$.
\end{prop}
\begin{proof}
	The adjunction $I_2^R \dashv I_3^R$ and the equality $I_2^R I_3^R = \text{Res}_{i_{\ca{A}}} \text{Res}_{i_{\ca{A}}^{R}} = \text{id}$ imply that $I_3^R$ is fully faithful.

	Then, the adjunction $I_2^R \dashv I_3^R$ implies $\derived(\ca{R}) = \langle I_3^R(\derived(\ca{A})), \ker I_2^{R} \rangle$.
	By \autoref{SOD-prop-KL} we know that $\ker I_2^{R} = I_1 ( \derived(\ca{B}))$, and the existence of the SOD follows.

	\begin{comment}
	Secondly, we need to prove semiorthogonality.
	We have $I_{1}^R I_3^R = \text{Res}_{i_{\ca{B}}} \text{Res}_{i_{\ca{A}}^{R}} = 0$, from which it follows $I_3^R(\derived(\ca{A})) \subset I_1(\derived(\ca{B}))^{\perp}$.
	As $I_3^R(\derived(\ca{A}))$ is left admissible and $I_1(\derived(\ca{B}))$ is right admissible, by \cite{BonKap89} we know that there exists a semiorthogonal decomposition
	\[
		\derived(\ca{R}) = \langle I_3^R(\derived(\ca{A})), {}^{\perp} I_3^R(\derived(\ca{A})) \cap I_1(\derived(\ca{B}))^{\perp}, I_1(\derived(\ca{B})) \rangle.
	\]
	However, as $I_2^R \dashv I_3^R$, being in ${}^{\perp} I_3^R(\derived(\ca{A}))$ is equivalent in being in $\text{Ker} I_2^R$, which in turn is equivalent to being in $I_2(\derived(\ca{A}))^{\perp}$.
	Thus, by \autoref{SOD-prop-KL}, we have ${}^{\perp} I_3^R(\derived(\ca{A})) \cap I_1(\derived(\ca{B}))^{\perp} = 0$, and we get the desired SOD.
	\end{comment}

	Using the relations $I_2^R I_3^R = \text{id}$ and $I_1^R I_3^R = 0$, we see from distinguished triangle \eqref{dt-sod-KL} that $I_1^L I_3^R = - \oL{A} \varphi$, which proves the claim for the left gluing functor.
\end{proof}

\begin{rmk}
	\label{rmk:to-cpt-objs-2}
	If the bimodule $\varphi$ is $\ca{B}$-perfect, we can apply \autoref{compact-obj-dec} to get a SOD of compact objects $\derived(\ca{R})^c = \langle I^R_3(\derived(\ca{A})^c), I_2(\derived(\ca{B})^c) \rangle$.
\end{rmk}

\begin{rmk}
	\label{rmk:SOD-for-all}
	Notice that even if the above result is proved for additive dg-categories, by the equivalence \eqref{modules-over-additive-closure} we get that, given any two small dg-categories $\ca{A}$, $\ca{B}$, we have SODs
	\[
		\derived(\ca{B} \sqcup_{\varphi} \ca{A}) = \langle \ind{i_{\ca{B}}}(\derived(\ca{B})), \ind{i_{\ca{A}}}(\derived(\ca{A})) \rangle = \langle I_4(\derived(\ca{A})), \ind{i_{\ca{B}}}(\derived(\ca{B})) \rangle,
	\]
	where $i_{\ca{B}}$ and $i_{\ca{A}}$ are defined in \eqref{eqn:inclusion-easy-gluing}, and for $M \in \derived(\ca{A})$, $a \in \ca{A}$, and $b \in \ca{B}$, we have $I_4(M)_{a} = M_a$ and $I_4(M)_{b} = 0$.
\end{rmk}

\begin{ex}
	Let us consider an upper triangular dg-algebra $R$ of the form \eqref{upper-triangular-dga} with $V$ a vector space concentrated in degree $0$ such that $\dim_k V < \infty$.

	Such (trivial) dg-algebra can be obtained as the path algebra of a quiver with two vertices and morphisms from the second to the first vertex indexed by elements of $V$:
	\[
		\begin{tikzcd}
			2 \ar[r, "V"] & 1.
		\end{tikzcd}
	\]
	Therefore, modules over $R$ give representations of the quiver.

	In this setup, \autoref{SOD-prop-KL} (resp. \autoref{SOD-prop-mine}) recovers the well known full exceptional collection given by the projective (resp. simple) modules.

	More precisely, \eqref{SOD-KL} reads $\derived(R) = \langle P_1, P_2 \rangle$ where $P_1 = k \otimes_{k_1} R = k_1 \oplus V$ and $P_2 = k \otimes_{k_2} R = k_2$ (here $k_i$ denotes the field acting on the $i$-th vertex in the path algebra interpretation), whereas \eqref{SOD-mine} reads $\derived(R) = \langle S_2, S_1 \rangle$ where $S_i = k_i$ is seen as a module via the projection $R \rightarrow k_i$.
\end{ex}

\begin{ex}
	\label{ex:inclusion-doesnt-preserve-cpt}
	We now give an example of a SOD of a cocomplete triangulated category in terms of cocomplete subcategories which does not induce a SOD of compact objects.
	Consider the upper triangular dg-algebra 
	\begin{equation}
		\label{upper-triangular-dga}
		R = \left(
			\begin{array}{cc}
				k & V\\
				0 & k
			\end{array}
		\right)
	\end{equation}
	where $V = k[q]$ is concentrated in degree $0$.
	To make things clear, we will denote the top left $k$ as $k_1$, and the bottom right $k$ as $k_2$.
	From \autoref{SOD-prop-mine}, \autoref{ex:upper-triangular-dg-algebra}, and \autoref{rmk:SOD-for-all} we know that there exists a SOD\footnote{Notice that we necessarily need to use this SOD, as by \autoref{rmk:to-cpt-objs} the one of \autoref{SOD-prop-KL} always induces a SOD of compact objects.} $\derived(R) = \langle \derived(k_1), \derived(k_2) \rangle$ with left gluing functor given by $- \stackrel{L}{\otimes}_{k} V : \derived(k_1) \rightarrow \derived(k_2)$.
	As $V$ is not perfect, \autoref{rmk:to-cpt-objs-2} does not apply and we cannot deduce a SOD for compact objects.
	Indeed, such a decomposition cannot exist because the inclusion of $\derived(k_1)$ does not preserve compact objects.
	To see this, consider the module $k_1 \in \derived(k_1)$ as a module over $R$ via the projection map $R \rightarrow k_1$.
	As a module over $k_1$, this module is compact.
	Let us consider the module $\bigoplus_{n \geq 0} k_2 \in \derived(k_2)$ as a module over $R$.
	Then, we have
	\[
		\Hom_{\derived(R)}(k_1, \bigoplus_{n \geq 0} k_2) \simeq \Hom_{\derived(k_2)}(\bigoplus_{n \geq 0} k_2, \bigoplus_{n \geq 0} k_2) \simeq \prod_{n \geq 0} \left( \bigoplus_{n \geq 0} k_2 \right),
	\]
	whereas
	\[
		\bigoplus_{n \geq 0} \Hom_{\derived(R)}(k_1, k_2) \simeq \bigoplus_{n \geq 0} \Hom_{\derived(k_2)}( \bigoplus_{n \geq 0} k_2 , k_2) \simeq \bigoplus_{n \geq 0} \left( \prod_{n \geq 0} k_2 \right),
	\]
	proving that $k_1$ is not compact in $\derived(R)$.
\end{ex}

\begin{rmk}
	The existence of the SOD in \autoref{SOD-prop-mine} is motivated by \cite[Theorem 3.15]{Halpern-Shipman16} and the discussion preceding it.
	Let us explain the difference between \eqref{SOD-KL} and \eqref{SOD-mine}.
	This is best understood looking at modules over rings.
	Assume we have two rings $R$ and $A$, and two morphisms $i: A \rightarrow R$, $g : R \rightarrow A$ such that $gi = \text{id}_A$.
	Starting from an $A$-module $N_A$ we can produce two different $R$-modules.
	Namely, we can either consider $N_A$ with the structure of $R$-module given by $g$, {\it i.e.} we {\it restrict} the action, or we can consider the $R$-module $N_A \otimes_A R$ (in this case we are {\it inducing} the action via $i$).
	The first construction corresponds to the functor $\text{Res}_g$, while the second one corresponds to $\text{Ind}_{i}$.
	Hence, in \eqref{SOD-KL} we are {\it inducing} the $\ca{R}$-module structure, whereas in \eqref{SOD-mine} we are {\it restricting} it. 
\end{rmk}

To conclude this section, let us notice that $\ca{R}^{\mathrm{opp}} \simeq \ca{A}^{\mathrm{opp}} \times_{\varphi^{\mathrm{opp}}} \ca{B}^{\mathrm{opp}}$, and therefore \autoref{SOD-prop-KL} implies the existence of a SOD $\derived(\ca{R}^{\mathrm{opp}}) = \langle J_1(\derived(\ca{A}^{\mathrm{opp}})), J_2(\derived(\ca{B}^{\mathrm{opp}})) \rangle$ with embedding functors given by $J_1 = \ind{i_{\ca{A}^{\mathrm{opp}}}}$ and $J_2 = \ind{i_{\ca{B}^{\mathrm{opp}}}}$, and right gluing functor $\varphi \oL{B} - $.
In this situation \autoref{compact-obj-dec-2} applies, and we get
\begin{equation}
	\label{SOD-left-modules}
	\derived(\ca{R}^{\mathrm{opp}})^c = \langle J_1(\derived(\ca{A}^{\mathrm{opp}})^c), J_2(\derived(\ca{B}^{\mathrm{opp}})^c) \rangle.
\end{equation}

In the same way, \autoref{SOD-prop-mine} implies $\derived(\ca{R}^{\mathrm{opp}}) = \langle J_3^R(\derived(\ca{B}^{\mathrm{opp}})), J_1(\derived(\ca{A}^{\mathrm{opp}})) \rangle$, where $J^R_3 = \res{i_{\ca{B}^{\mathrm{opp}}}^R}$.
Using \autoref{thm:adjoints} we see that, if $\varphi$ is $\ca{A}$-perfect, then \autoref{compact-obj-dec} applies, and we get
\begin{equation}
	\label{SOD-left-modules-2}
	\derived(\ca{R}^{\mathrm{opp}})^c = \langle J_3^R(\derived(\ca{B}^{\mathrm{opp}})^c), J_1(\derived(\ca{A}^{\mathrm{opp}})^c) \rangle.
\end{equation}

\subsection{Modules and bimodules on glued categories}
\label{subsect:modules-on-glued}

Let $\ca{A}$, $\ca{B}$ be two small dg-categories and consider $\varphi \in \bimod{A}{B}$ a bimodule.
For the rest of this section, $\ca{R}$ will denote either $\ca{B} \sqcup_{\varphi} \ca{A}$ or $\ca{B} \times_{\varphi[1]} \ca{A}$.

We will now describe the notation we will use to work with modules over $\ca{R}$.
We follow \cite[§ 7.2]{AL-P-n-func}.

A module $F \in \stmod{\ca{R}}$ can be described as a couple $\rmodule{F_{\ca{A}}}{F_{\ca{B}}}$ of an $\ca{A}$- (resp. $\ca{B}$-) module $F_{\ca{A}}$ (resp. $F_{\ca{B}}$) together with a closed, degree $0$ morphism $\rho_F  \in \mathrm{Hom}_{\ca{B}}(F_{\ca{A}} \otimes_{\ca{A}} \varphi, F_{\ca{B}})$ which we will call the {\it structure morphism}.
A morphism $\alpha : F \rightarrow G$ of degree $i$ is given by a couple of degree $i$ morphisms $(\alpha_{\ca{A}} : F_{\ca{A}} \rightarrow G_{\ca{A}}, \alpha_{\ca{B}} : F_{\ca{B}} \rightarrow G_{\ca{B}})$ such that the following diagram commutes
\begin{equation}
	\label{morphisms-right-modules}
	\begin{tikzcd}
		F_{\ca{A}} \otimes_{\ca{A}} \varphi \ar[d, "\alpha_{\ca{A}} \otimes \mathrm{id}"] \ar[r, "\rho_F"] & F_{\ca{B}} \ar[d, "\alpha_{\ca{B}}"]\\
		G_{\ca{A}} \otimes_{\ca{A}} \varphi \ar[r, "\rho_G"] & G_{\ca{B}}
	\end{tikzcd}.
\end{equation}
Differential and composition are computed componentwise.
Notice that this description of the category $\stmod{\ca{R}}$ mirrors the SOD of \autoref{SOD-prop-mine}.

\begin{rmk}
	\label{rmk:construction-module-from-matrix}
	When $\ca{R} = \ca{B} \times_{\varphi[1]} \ca{A}$ the recipe to produce an $\ca{R}$-module from the matrix $\rmodule{F_{\ca{A}}}{F_{\ca{B}}}$ and the structure morphism $\rho_F$ is the following: take $(b,a,\mu) \in \ca{R}$ with $\mu \in Z^0({}_a \varphi_b[1])$, then define the complex (notice that this construction is well defined because we have an explicit construction of the cone in ${\bf Mod}\text{-}k$)
	\[
		F(b,a,\mu) := \mathrm{cone} \left( (F_{\ca{A}})_a[-1] \xrightarrow{m \mapsto m \otimes \mu} (F_{\ca{A}} \otimes_{\ca{A}} \varphi)_b \xrightarrow{s \mapsto \rho_F(s)} F_b\right).
	\]
	To a degree $i$ morphism $(f_{\ca{B}}, f_{\ca{A}}, f_{\ca{B}\ca{A}}) : (b,a,\mu) \rightarrow (b',a',\mu')$ in $\ca{R}^{\mathrm{opp}}$ we associate the morphism given by the matrix $\bimodule{f_{\ca{B}}}{\phi}{0}{f_{\ca{A}}}$ with $\phi : (F_{\ca{A}})_a \rightarrow (F_{\ca{B}})_{b'}$ given by the formula $\phi(m) = \rho_{F}(m \otimes f_{\ca{B}\ca{A}})$.
\end{rmk}

The category $\stmod{R^{\mathrm{opp}}}$ admits a similar description: objects are couples\footnote{We think of them as column vectors.} $\rmodule{F_{\ca{A}}}{F_{\ca{B}}}^{t}$ together with a structure morphism $\rho_F : \varphi \otimes_{\ca{B}} F_{\ca{B}} \rightarrow F_{\ca{A}}$, and morphisms are couples making an analogue of \eqref{morphisms-right-modules} commute.

The category $\bimod{\ca{R}}{\ca{R}}$ can be described as follows: a bimodule is given by a matrix of bimodules
\[
	\bimodule{{}_{\ca{A}} F_{\ca{A}}}{{}_{\ca{A}} F_{\ca{B}}}{{}_{\ca{B}} F_{\ca{A}}}{{}_{\ca{B}} F_{\ca{B}}}
\]
together with a closed, degree $0$ structure morphism
\[
	\begin{aligned}
	\rho_F \in & \, \mathrm{Hom}_{\ca{A}\text{-}\ca{B}}({}_{\ca{A}} F_{\ca{A}} \otimes_{\ca{A}} \varphi, {}_{\ca{A}} F_{\ca{B}} ) \oplus \mathrm{Hom}_{\ca{B}\text{-}\ca{B}}({}_{\ca{B}} F_{\ca{A}} \otimes_{\ca{A}} \varphi, {}_{\ca{B}} F_{\ca{B}}) \oplus \\
	& \, \oplus \mathrm{Hom}_{\ca{A}\text{-}\ca{A}}(\varphi \otimes_{\ca{B}} ({}_{\ca{B}} F_{\ca{A}}), {}_{\ca{A}} F_{\ca{A}}) \oplus \mathrm{Hom}_{\ca{A}\text{-}\ca{B}}(\varphi \otimes_{\ca{B}} ({}_{\ca{B}} F_{\ca{B}}), {}_{\ca{A}} F_{\ca{B}})
	\end{aligned}
\]
whose components make the following diagram commute
\[
	\begin{tikzcd}
		\varphi \otimes_{\ca{B}} ({}_{\ca{B}} F_{\ca{A}}) \otimes_{\ca{A}} \varphi \ar[d] \ar[r] & \varphi \otimes_{\ca{B}} ({}_{\ca{B}} F_{\ca{B}}) \ar[d]\\
		{}_{\ca{A}} F_{\ca{A}} \otimes_{\ca{A}} \varphi \ar[r] & {}_{\ca{A}} F_{\ca{B}}.
	\end{tikzcd}
\]
Morphisms of degree $i$ are matrices of morphisms of degree $i$ that commute with the components of the structure morphism in way similar to \eqref{morphisms-right-modules}.
Differential and composition are computed componentwise.

The bar complex $\barr{\ca{R}}$ is described by the matrix
\begin{equation}
	\label{bar-complex}
	\left(
		\begin{array}{cc}
			\barr{A} & \left\{ \begin{tikzcd}[column sep = 4em] \barr{A} \otimes_{\ca{A}} \varphi \otimes_{\ca{B}} \barr{B} \ar[r, "{\lmodule{-\mathrm{id} \otimes \tau}{\tau \otimes \mathrm{id}}}"] & \underset{\mathrm{deg. 0}}{(\barr{A} \otimes_{\ca{A}} \varphi) \oplus (\varphi \otimes_{\ca{B}} \barr{B})} \end{tikzcd} \right\}\\
			0 & \barr{B}
		\end{array}
	\right)
\end{equation}
whose structure morphism components are given in \cite[(7.16)]{AL-P-n-func}.

Using \eqref{bar-complex} one can give a description of $\barmod{\ca{R}}$ and $\barbimod{\ca{R}}{\ca{R}}$.
Objects are the same as before for both categories, however, as the structure morphism is defined using the standard tensor product and not the bar tensor product, we will adopt the following convention: we will define structure morphisms using the bar tensor product, but we will ensure that these morphisms lie in the image of the faithful functor $\stmod{B} \hookrightarrow \barmod{B}$ constructed in \cite[Proposition 3.3]{Anno-Logvinenko-Bar-Categories}.
A morphism $\alpha : F \rightarrow G$ of degree $i$ in $\barmod{\ca{R}}$ is a triple $(\alpha_{\ca{A}}, \alpha_{\ca{B}}, \alpha_{\ca{A}\ca{B}})$ of degree $(i,i,i-1)$ morphisms
\[
	\begin{tikzcd}
		F_{\ca{A}} \ar[r, "\alpha_{\ca{A}}"] & G_{\ca{A}} & F_{\ca{B}} \ar[r, "\alpha_{\ca{B}}"] & G_{\ca{B}} & F_{\ca{A}} \ob{A} \varphi \ar[r, dashed, "\alpha_{\ca{A}\ca{B}}"] & G_{\ca{B}}
	\end{tikzcd}
\]
in $\barmod{A}$, $\barmod{B}$, and $\barmod{B}$, respectively.
Composition law and differential can be found in \cite[§ 7.2]{AL-P-n-func}.
Here we only notice that a closed, degree $0$ morphism in $\barmod{\ca{R}}$ is the choice of two morphisms $\alpha_{\ca{A}}$, $\alpha_{\ca{B}}$ that make \eqref{morphisms-right-modules} commute up to the homotopy $\alpha_{\ca{A}\ca{B}}$.
Morphisms change in a rather complicated way for the bimodule category and we refer to \cite[pag. 55]{AL-P-n-func} for a description.

%
%
%
% SPHERICAL FUNCTORS
%
%
%

\subsection{Spherical functors}

Given an $\ca{A}\text{-}{\ca{B}}$ bimodule which is $\ca{A}$- and $\ca{B}$-perfect, we know that the functor $f_M = - \stackrel{L}{\otimes}_{\ca{A}} M \colon \derived(\ca{A}) \rightarrow \derived(\ca{B})$ has left and right adjoints $f_M^L$, $f_M^R$.
Moreover, we know that all these functors are induced by bimodules.
Therefore, it makes sense to give the following

\begin{dfn}
	The {\it twist} associated to $M$ is the endofunctor $t_M$ of $\derived(\ca{B})$ defined as the cone of the counit
	\[
		f_M f_M^R \rightarrow \mathrm{id} \rightarrow t_M.
	\]
	The {\it cotwist} associated to $M$ is the endofunctor $c_M$ of $\derived(\ca{A})$ defined as the shifted cone of the unit
	\[
		c_M \rightarrow \mathrm{id} \rightarrow f_M^R f_M.
	\]
\end{dfn}

To get a functorial construction of the cones, we can take the cones of the maps between the bimodules inducing the above functors.

\begin{dfn}
	\label{def-spherical-functor}
	The bimodule $M$ is called a {\it spherical bimodule}, and the functor $f_M$ is called {\it spherical}, if all the following hold
	\begin{itemize}\tabularnewline
		\item $t_M$ is an autoequivalence of $\derived(\ca{B})$.
		\item $c_M$ is an autoequivalence of $\derived(\ca{A})$.
		\item The natural morphism $f^L_M t_M[-1] \rightarrow f^L_M f_M f_M^R \rightarrow f_M^R$ is an isomorphism.
		\item The natural morphism $f_M^R \rightarrow f_M^R f_M f_M^L \rightarrow c_M f^L[1]$ is an isomorphism.
	\end{itemize}
\end{dfn}

\begin{thm}[\text{\cite[Theorem 5.1]{Anno-Log-17}}]
	\label{thm:spherical-functors}
	If any two conditions of \autoref{def-spherical-functor} are satisfied, then all four are satisfied.
\end{thm}

\begin{dfn}
	\label{dfn-sigma}
	For any spherical bimodule $M \in \barbimod{A}{B}$ we fix once and for all a bimodule $C_M \in \barbimod{A}{A}$ that lifts the cotwist $c_{M}$ and a lift
	\begin{equation}
		\label{eqn:sigma-dfn}
			\sigma: M \ob{B} \dual{M}{B}[-1] \rightarrow C_M
	\end{equation}
	in $\barbimod{A}{A}$ of the map $M \oL{B} M^{\widetilde{\ca{B}}} \rightarrow C_M$ that fits in the distinguished triangle
	\[
		C_M \rightarrow \ca{A} \rightarrow M \oL{B} M^{\widetilde{\ca{B}}}.
	\]
\end{dfn}

In the following, when we write $\sigma$ we will always refer to the map \eqref{eqn:sigma-dfn} for some spherical bimodule which will be clear from the context.

\begin{rmk}
	\label{rmk:cotwist-identifies-maps}
	For later reference, we notice that under the isomorphism $f_M^R \simeq c_M f^L[1]$ the counit map $f^L_M f \rightarrow \mathrm{id}$ is identified with the map $f_M^R f \xrightarrow{\sigma} c_M[1]$.
\end{rmk}

%GLUING SPHERICAL FUNCTORS
%%%
%%%
%%%  THE GENERAL CASE
%%%
%%%
%%%
\section{Gluing spherical functors}
\label{section:general-case}

The following theorem is the main result of this article.

\begin{thm}
	\label{twist-twist=twist}
	Let $\ca{A}$, $\ca{B}$ and $\ca{C}$ be three small dg-categories over a field $k$, and $M \in \derived(\ca{A}\text{-}\ca{C})$, $N \in \derived(\ca{B}\text{-}\ca{C})$ be two  spherical bimodules.
	Then, the bimodule $P := \rmodule{M}{N}^t \in \derived((\ca{B} \sqcup_{\varphi} \ca{A})\text{-}\ca{C})$, $\varphi = \textup{RHom}_{\ca{C}}(N,M)$, with structure morphism
	\[
		\mathrm{RHom}_{\ca{C}}(N,M) \oL{B} N \xrightarrow{\mathrm{tr}} M
	\]
	is spherical, the twist around it is given by the composition $t_N \circ t_M$, and the cotwist is described by the matrix
	\[
		C_P = \left(
		\begin{array}{cc}
			C_M & 0 \\
			\mathrm{RHom}_{\ca{C}}(M,N)[-1] & C_N
		\end{array}
		\right)
		\in \derived((\ca{B} \sqcup_{\varphi} \ca{A})\text{-}(\ca{B} \sqcup_{\varphi} \ca{A}))
	\]
	with non-zero structure morphisms
	\[
		\begin{aligned}
			& \, \mathrm{RHom}_{\ca{C}}(M,N)[-1] \oL{A} \mathrm{RHom}_{\ca{C}}(N,M) \xrightarrow{\sigma \circ \mathrm{cmps}} C_N\\
			& \, \mathrm{RHom}_{\ca{C}}(N,M) \oL{B}  \mathrm{RHom}_{\ca{C}}(M,N)[-1] \xrightarrow{\sigma \circ \mathrm{cmps}} C_M,
		\end{aligned}
	\]
	where $\mathrm{cmps}$ denotes the composition of morphisms.
\end{thm}

\begin{rmk}
	Notice that as $\text{RHom}_{\ca{C}}(N,M)$ is only defined up to quasi-isomorphism, also $\ca{B} \sqcup_{\varphi} \ca{A}$ is only defined up to a quasi-equivalence.
	However, as the statement of the theorem concerns $P$ as an element of $\derived((\ca{B} \sqcup_{\varphi} \ca{A})\text{-}\ca{C})$, this does not create any problem.
\end{rmk}

By \autoref{prop:additive-not-restrictive} it is not restrictive to assume that $\ca{A}$ and $\ca{B}$ are additive, and to prove the result for the category $\ca{R} := \ca{B} \times_{\varphi[1]} \ca{A}$, $\varphi = M \ob{C} \dual{N}{C}$, and the bimodule $P: = \rmodule{M}{N}^t$ with structure morphism
\begin{equation}
	\label{structure-morphism-P}
	\rho : M \ob{C} \dual{N}{C} \ob{B} N \xrightarrow{\mathrm{id} \ob{} \mathrm{tr}} M.
\end{equation}
In the following we will stick to this setup.

As a corollary to \autoref{twist-twist=twist}, we readily get

\begin{cor}
	\label{cor:commutativity-relation}
	With the same setting as in \autoref{twist-twist=twist}, we have
	\[
		t_{N} \circ t_M \simeq t_{t_N(M)} \circ t_{N},
	\]
	where $t_N(M) \in \derived(\ca{A}\text{-}\ca{C})$ is the bimodule given by ${}_a t_N(M)_{c} = t_N({}_a M)_c$ for any $a \in \ca{A}$, $c \in \ca{C}$.
\end{cor}

\begin{proof}
	We will see in \autoref{rmk:HLS-theorem} that we have
	\[
		\derived(\ca{R}) =  \langle c_P(\ind{i_{\ca{A}}}(\derived(\ca{A}))), \ind{i_{\ca{B}}}(\derived(\ca{B})) \rangle = \langle \res{i^R_{\ca{A}}}(\derived(\ca{A})), \ind{i_{\ca{B}}}(\derived(\ca{B})) \rangle.
	\]
	Moreover, combining \eqref{SOD-KL} with the fact that $c_P$ is an autoequivalence, we get
	\[
		\derived(\ca{R}) = \langle c_P(\ind{i_{\ca{B}}}(\derived(\ca{B}))), c_P(\ind{i_{\ca{A}}}(\derived(\ca{A}))) \rangle,
	\]
	and therefore we can apply \cite[Theorem 4.14]{Halpern-Shipman16}.
	We deduce that $E = \ca{A}_{i^{R}_{\ca{A}}} \ob{R} P$ is a spherical bimodule, and that $t_P \simeq t_{E} \circ t_N$.

	To conclude, notice that we have an isomorphism in $\derived(\ca{A}\text{-}\ca{C})$
	\[
		E \simeq \left\{ M \ob{C} \dual{N}{C} \ob{B} N \xrightarrow{\mathrm{id} \ob{} \mathrm{tr}} \underset{\deg 0}{M} \right\} \simeq t_N(M)
	\]
	yielding the claim.
\end{proof}

We split the proof \autoref{twist-twist=twist} in four parts: the proof that $P$ is perfect on both sides (\autoref{lem:P-is-perfect}), the description of the twist (\autoref{prop:description-twist}), the description of the cotwist (\autoref{prop:description-cotwist}), and the proof that the cotwist is an autoequivalence (\autoref{prop:cotwist-autoequivalence}).
Then, we will conclude using \autoref{thm:spherical-functors}.

For future reference, we denote
\begin{equation}
	\label{functor-f-P}
	\derived(\ca{R}) \xrightarrow{f_P} \derived(\ca{C})
\end{equation}
the functor induced by $P$.

We begin by proving the following

\begin{lem}
	\label{lem:P-is-perfect}
	The $\ca{R}\text{-}\ca{C}$ bimodule $P$ of \autoref{twist-twist=twist} is $\ca{R}$- and $\ca{C}$-perfect.
\end{lem}

\begin{proof}
	Applying the construction of \autoref{rmk:construction-module-from-matrix} to $\ca{R}^{\mathrm{opp}}$, we see that a left $\ca{R}$-module is $\ca{C}$-perfect if and only if its components are.
	Therefore, $P$ is $\ca{C}$-perfect.

	Recall that $\ca{R}$ is defined as the gluing of $\ca{B}$ and $\ca{A}$ along $M \ob{C} \dual{N}{C}[1]$.
	We claim that the bimodule $M \ob{C} \dual{N}{C}$ is $\ca{A}$-perfect, and therefore $\ca{R}$-perfectness of $P$ follows from \eqref{SOD-left-modules-2} and left perfectness of its components.

	By \autoref{thm:adjoints} and $\ca{A}$-perfectness of $M$, to prove that $M \ob{C} \dual{N}{C}$ is $\ca{A}$-perfect it is enough to prove that $\dual{N}{C}$ is $\ca{C}$-perfect.
	However, $\dual{N}{C} = \mathrm{Hom}_{\ca{C}}(N \oo{C} \barr{C}, \ca{C})$ is the dual of a $\ca{C}$ h-projective, $\ca{C}$-perfect module, and therefore it is still $\ca{C}$ h-projective and $\ca{C}$-perfect, see \cite[pag. 10]{Anno-Log-17} for an explanation of this fact.
\end{proof}

\subsection{The twist}

\begin{prop}
	\label{prop:description-twist}
	The twist around the functor \eqref{functor-f-P} is given by $t_N \circ t_M$.
\end{prop}

\begin{proof}
	Notice that convolution shifted by $-1$ of following diagram in $\bimod{R}{R}$ (where $\mathrm{cmps}$ denotes the composition in the dg-category $\ca{R}$)
	\begin{equation}
	\label{diagonal-bimodule-h-proj}
	\begin{tikzcd}[column sep = 10em]
		\ca{R} \oo{A} \barr{A} \oo{A} \varphi \oo{B} \barr{B} \oo{B} \ca{R} \ar[d, "{-(\mathrm{id}^{\otimes 2} \otimes \mathrm{cmps}) \circ (\mathrm{id}^{\otimes 3} \otimes \tau \otimes \mathrm{id})}"] \ar[r, "{(\mathrm{cmps} \otimes \mathrm{id}^{\otimes 2}) \circ (\mathrm{id} \otimes \tau \otimes \mathrm{id}^{\otimes 3})}"] & \ca{R} \oo{B} \barr{B} \oo{B} \ca{R}\\
		\ca{R}  \oo{A} \barr{A} \oo{A} \ca{R}
	\end{tikzcd}
	\end{equation}
	is h-projective by \cite[Proposition 2.5]{Anno-Log-17} and quasi isomorphic to the diagonal bimodule by the distinguished triangle of \cite[Proposition 4.9]{KL15} (see also the proof of \cite[Proposition 3.11]{Lunts-Cat-res-sing}).
	Hence, we can use \eqref{diagonal-bimodule-h-proj} to compute $\dual{P}{C} \ob{R} P$.

	Notice that the bimodule $\dual{P}{C}$ is given by the matrix $\rmodule{\dual{M}{C}}{\dual{N}{C}}$ with structure morphism $\dual{M}{C} \ob{A} M \ob{C} \dual{N}{C} \xrightarrow{\mathrm{tr} \ob{} \mathrm{id}} \dual{N}{C}$.
	Hence, the bimodule $\dual{P}{C} \ob{R} P$ is homotopy equivalent to the convolution shifted by $-1$ of the following diagram in $\barbimod{R}{R}$
	\[
		\begin{tikzcd}[column sep = 10em]
			\dual{M}{C} \ob{A} M \ob{C} \dual{N}{C} \ob{B} N \ar[d, "{-\mathrm{id}^{\ob{} 2} \ob{} \mathrm{tr}}"] \ar[r, "{\mathrm{tr} \ob{} \mathrm{id}^{\ob{} 2}}"]&  \dual{N}{C} \ob{B} N\\
			\dual{M}{C} \ob{A} M
		\end{tikzcd}.
	\]
	The cone of the counit $\dual{P}{C} \ob{R} P \rightarrow \ca{C}$ is homotopy equivalent to the convolution of the following diagram in $\barbimod{R}{R}$
	\begin{equation}
		\label{description-twist}
		\begin{tikzcd}[column sep = 10em]
			\dual{M}{C} \ob{A} M \ob{C} \dual{N}{C} \ob{B} N \ar[d, "{-\mathrm{id}^{\ob{} 2} \ob{} \mathrm{tr}}"] \ar[r, "{\mathrm{tr} \ob{} \mathrm{id}^{\ob{} 2}}"]&  \dual{N}{C} \ob{B} N \ar[d, "\mathrm{tr}"]\\
			\dual{M}{C} \ob{A} M \ar[r, "\mathrm{tr}"] & \ca{C}
		\end{tikzcd}.
	\end{equation}
	However, we have an homotopy equivalence
	\[
		\eqref{description-twist} \simeq \left\{ \dual{M}{C} \ob{A} M \xrightarrow{\mathrm{tr}} \ca{C} \right\} \ob{C} \left\{ \dual{N}{C} \ob{B} N \xrightarrow{\mathrm{tr}} \ca{C}\right\}
	\]
	by \cite[Lemma 3.43]{Anno-Logvinenko-Bar-Categories}, and the latter bimodule represents the composition $t_N \circ t_M$.
	Hence we get the claimed description of the twist.
\end{proof}
\subsection{The cotwist}

\begin{prop}
	\label{prop:description-cotwist}
	The cotwist around the functor \eqref{functor-f-P} is described by the following matrix
	\begin{equation}
		\label{cotwist-bimodule-description}
		C_P = \left(
		\begin{array}{cc}
			C_M & 0 \\
			N \ob{C} \dual{M}{C} [-1] & C_N
		\end{array}
		\right)
		\in \derived(\ca{R}\text{-}\ca{R})
	\end{equation}
	with non-zero structure morphisms
	\[
		\begin{tikzcd}[column sep = 8em, row sep = tiny]
			N \ob{C} \dual{M}{C} [-1] \ob{A} M \ob{C} \dual{N}{C} \ar[r, "{\sigma \circ (\mathrm{id} \ob{} \mathrm{tr} \ob{} \mathrm{id})}"] & C_N\\
			M \ob{C} \dual{N}{C} \ob{A} N \ob{C} \dual{M}{C}[-1] \ar[r, "{\sigma \circ (\mathrm{id} \ob{} \mathrm{tr} \ob{} \mathrm{id})}"] & C_M.
		\end{tikzcd}
	\]
\end{prop}

\begin{proof}
	We have
	\[
		P \ob{C} \dual{P}{C} = \bimodule{M \ob{C} \dual{M}{C}}{M \ob{C} \dual{N}{C}}{N \ob{C} \dual{M}{C}}{N \ob{C} \dual{N}{C}}
	\]
	with structure morphisms
	\[
		\begin{array}{lr}
			\begin{tikzcd}[column sep = 4em]
				M \ob{C} \dual{M}{C} \ob{A} M \ob{C} \dual{N}{C} \ar[r, "\mathrm{id} \ob{} \mathrm{tr} \ob{} \mathrm{id}"] & M \ob{C} \dual{N}{C}
			\end{tikzcd} & 
			\begin{tikzcd}[column sep = 4em]
				N \ob{C} \dual{M}{C} \ob{A} M \ob{C} \dual{N}{C} \ar[r, "\mathrm{id} \ob{} \mathrm{tr} \ob{} \mathrm{id}"] & N \ob{C} \dual{N}{C}
			\end{tikzcd}\\
			\begin{tikzcd}[column sep = 4em]
				M \ob{C} \dual{N}{C} \ob{B} N \ob{C} \dual{M}{C} \ar[r, "\mathrm{id} \ob{} \mathrm{tr} \ob{} \mathrm{id}"] & M \ob{C} \dual{M}{C}
			\end{tikzcd} &
			\begin{tikzcd}[column sep = 4em]
				M \ob{C} \dual{N}{C} \ob{B} N \ob{C} \dual{N}{C} \ar[r, "\mathrm{id} \ob{} \mathrm{tr} \ob{} \mathrm{id}"] & M \ob{C} \dual{N}{C}
			\end{tikzcd}
		\end{array}
	\]

	Denote $C_P'$ the bimodule in the statement of the proposition.
	We aim to show that the morphism
	\[
		P \ob{C} \dual{P}{C}
		\xrightarrow{\left(
		\begin{array}{cc}
			\sigma & 0\\
			\mathrm{id} & \sigma
		\end{array}
		\right)}
		C_P'[1]
	\]
	in $\barbimod{R}{R}$ identifies $C_P'[1]$ with the cone of the action morphism $\ca{R} \xrightarrow{\mathrm{act}} P \ob{C} \dual{P}{C}$.

	Using the description of $\barr{R}$ given in \eqref{bar-complex}, we can describe the action map for $P$ in $\barbimod{R}{R}$.
	As the diagonal bimodule has no $\ca{B}\text{-}\ca{A}$ component, to give a map $\ca{R} \rightarrow P \ob{C} \dual{P}{C}$ in $\barbimod{R}{R}$ is enough to give a matrix of maps between the various components plus two degree $-1$ maps
	\[
		\begin{array}{lcr}
			\begin{tikzcd}[ampersand replacement = \&]
				\barr{A} \ob{A} M \ob{C} \dual{N}{C} \ar[r, dashed] \& M \ob{C} \dual{N}{C}
			\end{tikzcd}
			& \mathrm{and} &
			\begin{tikzcd}[ampersand replacement = \&]
				M \ob{C} \dual{N}{C} \ob{B} \barr{B} \ar[r, dashed] \& M \oo{C} \dual{N}{C}
			\end{tikzcd}
		\end{array}
	\]
	making the relevant diagrams commutes, see \cite[§ 7.2]{AL-P-n-func}, and in particular diagram (7.17) in \emph{ibidem}, for an explanation.

	The action map for $P$ is given by the by the matrix $\bimodule{\mathrm{act}}{\mathrm{\alpha}}{0}{\mathrm{act}}$, where $\alpha$ is map given by the map of twisted complexes in $\barbimod{A}{B}$
	\[
		\begin{tikzcd}[column sep = 6em]
			\barr{A} \oo{A} M \ob{C} \dual{N}{C} \oo{B} \barr{B} \ar[r, "{\lmodule{- \mathrm{id} \oo{} \tau}{\tau \oo{} \mathrm{id}}}"] & \underset{\mathrm{deg. 0}}{(\barr{A} \oo{A} M \ob{C} \dual{N}{C}) \oplus (M \ob{C} \dual{N}{C} \oo{B} \barr{B})} \ar[d, "\tau \otimes \mathrm{id} + \mathrm{id} \otimes \tau"]\\
			{} & M \ob{C} \dual{N}{C},
		\end{tikzcd}
	\]
	together with the maps of degree $-1$
	\[
		\begin{tikzcd}[column sep = 8em]
			\barr{A} \ob{A} M \ob{C} \dual{N}{C} \ar[r, dashed, "{- (\chi_{\ca{A}} \ob{} \mathrm{id}) \circ (\tau \ob{} \mathrm{id}^{\ob{} 2})}"] & M \ob{C} \dual{N}{C}
		\end{tikzcd}
	\]
	and
	\[
		\begin{tikzcd}[column sep = 8em]
			M \ob{C} \dual{N}{C} \ob{B} \barr{B} \ar[r, dashed, "{(\mathrm{id} \ob{} \chi'_{\ca{B}}) \circ (\mathrm{id}^{\ob{} 2} \ob{} \tau)}"] & M \oo{C} \dual{N}{C}
		\end{tikzcd}
	\]
	where we are using the notation of \cite[Proposition 4.6]{Anno-Logvinenko-Bar-Categories}.
	
	Given this description of the action morphism, the statement of the proposition follows.
\end{proof}

Notice that the above description of the cotwist tells us that the following diagrams commute
\begin{equation}
	\label{cotwist-restricted-to-a}
	\begin{array}{lcr}
	\begin{tikzcd}
		\derived(\ca{A}) \ar[d, "\ind{i_{\ca{A}}}"'] \ar[r, "c_M"] & \derived(\ca{A}) \ar[d, "\res{i^{R}_{\ca{A}}}"]\\ 
		\derived(\ca{R}) \ar[r, "c_P"] & \derived(\ca{R})
	\end{tikzcd}
	& &
	\begin{tikzcd}
		\derived(\ca{B}) \ar[d, "\ind{i_{\ca{B}}}"'] \ar[r, "c_N"] & \derived(\ca{B})\\ 
		\derived(\ca{R}) \ar[r, "c_P"] & \derived(\ca{R}) \ar[u, "\res{i_{\ca{B}}}"']
	\end{tikzcd}
	\end{array}.
\end{equation}
Indeed, the following maps are homotopy equivalences in $\barmod{R}$ and $\barmod{B}$, respectively
\begin{equation}
	\label{homotopy-equiv-ind-a}
	\begin{aligned}
		F_{\ca{A}} \ob{A} \ca{R} \ob{R} C_P \oo{R} \barr{R} \xrightarrow{\simeq} & F_{\ca{A}} \oo{A} \barr{A} \oo{A} \barr{R} \oo{R} C_P \oo{R} \barr{R}\xrightarrow{\mathrm{id}^{\otimes 2} \otimes \tau \otimes \mathrm{id} \otimes \tau}\\
		\rightarrow & F_{\ca{A}} \ob{A} C_P = \rmodule{F_{\ca{A}} \ob{A} C_M}{0}.
	\end{aligned}
\end{equation}
and
\begin{equation*}
	\begin{aligned}
		F_{\ca{B}} \ob{B} \ca{R} \ob{R} C_P \oo{B} \barr{B} \xrightarrow{\simeq} & F_{\ca{B}} \oo{B} \barr{B} \oo{B} \barr{R} \oo{R} C_P \oo{B} \barr{B}\xrightarrow{\mathrm{id}^{\otimes 2} \otimes \tau \otimes \mathrm{id} \otimes \tau}\\
		\rightarrow & F_{\ca{B}} \ob{B} (C_P)_{i_{\ca{B}}} = \rmodule{0}{F_{\ca{B}} \ob{B} C_N}.
	\end{aligned}
\end{equation*}

\begin{rmk}
	\label{rmk:HLS-theorem}
	The commutativity of the left square in \eqref{cotwist-restricted-to-a}, \autoref{SOD-prop-mine}, and the fact that $c_M$ is an autoequivalence tell us that there exists a SOD
	\[
		\derived(\ca{R}) = \langle c_P(\ind{i_{\ca{A}}}(\derived(\ca{A}))), \ind{i_{\ca{B}}}(\derived(\ca{B})) \rangle.
	\]
	Therefore, we see that for the spherical functor induced by the bimodule $P$ the hypotheses of \cite[Theorem 4.14]{Halpern-Shipman16} are satisfied.
	However, let us remark that we cannot use the theorem from {\it ibidem} to prove \autoref{twist-twist=twist} because in {\it ibidem} the authors assume that the functor is spherical to deduce that the twist around it splits, whereas we use the description of the twist to prove that the functor is spherical.
\end{rmk}

\begin{prop}
	\label{prop:cotwist-autoequivalence}
	The cotwist around the functor \eqref{functor-f-P} is an autoequivalence of $\derived(\ca{R})$.
\end{prop}

\begin{proof}
	Let us assume that the cotwist is fully faithful, then we show that it is an autoequivalence.
	We will prove that the right orthogonal to the essential image of $c_P$ is zero, this will suffice because we know that $c_P$ has a right adjoint $c_P^R$, and the fully faithfulness assumption implies $\derived(\ca{R}) = \langle \ker (c_P^{R}), \im (c_P) \rangle$, where $\im(c_P)$ denotes the essential image of $c_P$.

	Take $G = \rmodule{G_{\ca{A}}}{G_{\ca{B}}} \in \im (c_P) ^{\perp}$.
	By \eqref{cotwist-restricted-to-a} we know that $\res{i^{R}_{\ca{A}}}(\derived(\ca{A})) \subset \im (c_P)$.
	Therefore, we get $G \in \res{i^{R}_{\ca{A}}}(\derived(\ca{A}))^{\perp}$, which, together with \autoref{SOD-prop-mine} and the right diagram in \eqref{cotwist-restricted-to-a}, implies for any $T_{\ca{B}} \in \derived(\ca{B})$
	\[
		0 = \mathrm{Hom}_{\derived(\ca{R})}(c_P(\ind{i_{\ca{B}}}(T_{\ca{B}})), G) \simeq \mathrm{Hom}_{\derived(\ca{B})}(c_N(T_{\ca{B}}), G_{\ca{B}}).
	\]
	As $c_N$ is an autoequivalence, we get $G_{\ca{B}} \simeq 0$.
	To conclude, we use that $\rmodule{G_{\ca{A}}}{0} \in \res{i^{R}_{\ca{A}}}(\derived(\ca{A}))^{\perp}$, and therefore $G_{\ca{A}} \simeq 0$.

	We now show that $c_P$ is fully faithful.
	First, notice that by the left square in \eqref{cotwist-restricted-to-a}, \autoref{SOD-prop-KL}, and \autoref{SOD-prop-mine}, we get that $c_P \vert_{\ind{i_{\ca{A}}}(\derived(\ca{A}))}$ is fully faithful.

	Next, take $F_{\ca{B}} \in \derived(\ca{B})$, $G_{\ca{A}} \in \derived(\ca{A})$ and consider $F = \ind{i_{\ca{B}}}(F_{\ca{B}}) = \rmodule{0}{F_{\ca{B}}}$ and $G = \ind{i_{\ca{A}}} (G_{\ca{A}}) = \rmodule{G_{\ca{A}} \oo{A} \barr{A}}{G_{\ca{A}} \ob{A} M \ob{C} \dual{N}{C}}$, where the structure morphism of the latter is given by the identity.
	Then, a morphism $\alpha : F \rightarrow G$ is given by a closed, degree $0$ morphism $\alpha_{\ca{B}} : F_{\ca{B}} \rightarrow G_{\ca{A}} \ob{A} M \ob{C} \dual{N}{C}$ in $\barmod{B}$. 
	The action of $C_P$ is given by
	\[
			F \ob{R} C_P \simeq \rmodule{F_{\ca{B}} \ob{B} N \ob{C} \dual{M}{C}[-1]}{F_{\ca{B}} \ob{B} C_N}
	\]
	with structure morphism
	\[
		F_{\ca{B}} \ob{B} N \ob{C} \dual{M}{C}[-1] \ob{A} M \ob{C} \dual{N}{C} \xrightarrow{(\mathrm{id} \ob{} \sigma) \circ (\mathrm{id}^{\ob{} 2} \ob{} \mathrm{tr} \ob{} \mathrm{id})} F_{\ca{B}} \ob{B} C_N,
	\]
	while $G \ob{R} C_P$ is homotopy equivalent to $\rmodule{G_{\ca{A}} \ob{A} C_M}{0}$ with a homotopy equivalence given by \eqref{homotopy-equiv-ind-a}.
	The morphism $\alpha$ is sent to the the morphism
	\begin{equation}
		\label{morphism-b-to-a}
		\begin{tikzcd}[column sep = 6em]
			F_{\ca{B}} \ob{B} N \ob{C} \dual{M}{C}[-1] \ar[r, "\alpha_{\ca{B}} \ob{} \mathrm{id}^{\ob{} 2}"] & G_{\ca{A}} \ob{A} M \ob{C} \dual{N}{C} \ob{B} N \ob{C} \dual{M}{C}[-1] \ar[d, "{(\mathrm{id} \ob{} \sigma) \circ (\mathrm{id}^{\ob{} 2} \ob{} \mathrm{tr} \ob{} \mathrm{id})}"']\\
			{} & G_{\ca{A}} \ob{A} C_M
		\end{tikzcd}
	\end{equation}

	A morphism $\beta : F \ob{R} C_P \rightarrow G \ob{R} C_P$ is given by a triple $(\beta_{\ca{A}}, 0, 0)$, where $\beta_{\ca{A}}$ is a closed, degree $0$ morphism $\beta_{\ca{A}} : F_{\ca{B}} \ob{B} N \ob{C} \dual{M}{C}[-1] \rightarrow G_{\ca{A}} \ob{A} C_M$ in $\barmod{A}$.
	However, using that $C_M$ induces an autoequivalence, that $\dual{M}{C}$ is homotopy equivalent to $\dual{M}{A} \ob{C} C_M[1]$ (and that under this homotopy equivalence the counit for $- \ob{C} \dual{M}{A} \dashv - \ob{A} M$ is identifed with $\sigma$ up to homotopy, see \autoref{rmk:cotwist-identifies-maps}), and the adjunctions $- \ob{B} N \dashv - \ob{C} \dual{N}{C}$, $- \ob{C} \dual{M}{A} \dashv - \ob{A} M$, we see that any such $\beta_{\ca{A}}$ comes (up to homotopy) from a unique (up to homotopy) morphism $\alpha_{\ca{B}}$ via \eqref{morphism-b-to-a}.
	Therefore, $c_P$ is fully faithful on morphisms from $\ind{i_{\ca{B}}}(\derived(\ca{B}))$ to $\ind{i_{\ca{A}}}(\derived(\ca{A}))$.

	Let us now consider $F = \rmodule{0}{F_{\ca{B}}}$ and $G = \rmodule{0}{G_{\ca{B}}}$.
	Then, a morphism $\alpha : F \rightarrow G$ is given by a triple $(0, \alpha_{\ca{B}}, 0)$, where $\alpha_{\ca{B}} : F_{\ca{B}} \rightarrow G_{\ca{B}}$ is a closed, degree $0$ mophism in $\barmod{B}$.
	Via $C_P$ this morphism is sent to the triple $( \alpha_{\ca{B}} \ob{} \mathrm{id}, \alpha_{\ca{B}} \ob{} \mathrm{id}, 0)$.
	In general, a morphism $\beta : F \ob{R} C_P \rightarrow G \ob{R} C_P$ is given by a triple of morphisms $(\beta_{\ca{A}}, \beta_{\ca{B}}, \beta_{\ca{A}\ca{B}})$ such that the following is a closed, degree $0$ morphism of twisted complexes in $\barmod{B}$
	\begin{equation*}
		%\label{morphism-b-to-b}
		\begin{tikzcd}
			F_{\ca{B}} \ob{B} N \ob{C} \dual{M}{C}[-1] \ob{A} M \ob{C} \dual{N}{C} \ar[r] \ar[rd, dashed, "\beta_{\ca{A}\ca{B}}"] \ar[d, "\beta_{\ca{A}} \ob{} \mathrm{id}"] & F_{\ca{B}} \ob{C} C_N \ar[d, "\beta_{\ca{B}} \ob{} \mathrm{id}"]\\
			G_{\ca{B}} \ob{B} N \oo{C} \dual{M}{C}[-1] \ob{A} M \ob{C} \dual{N}{C} \ar[r] & G_{\ca{B}} \ob{C} C_N
		\end{tikzcd}
	\end{equation*}
	Using adjunction, that $C_N$ induces an autoequivalence, and that $N$ is a spherical bimodule, we can deduce that $\beta$ comes (up to homotopy) from a unique (up to homotopy) $\alpha_{\ca{B}}$.
	Therefore, $c_P$ is fully faithful for morphisms from $\ind{i_{\ca{B}}}(\derived(\ca{B}))$ to $\ind{i_{\ca{B}}}(\derived(\ca{B}))$.

	Finally, consider $F = \rmodule{0}{F_{\ca{B}}}$ and $G = \rmodule{G_{\ca{A}} \oo{A} \barr{A}}{G_{\ca{A}} \ob{A} M \ob{C} \dual{N}{C}}$.
	There are no morphisms in $\derived(\ca{R})$ from $G$ to $F$, and thus we have to prove that any morphism $\beta : G \ob{R} C_P \rightarrow F \ob{R} C_P$ is homotopic to zero.
	The morphism $\beta$ is given by a triple $(\beta_{\ca{A}}, 0, \beta_{\ca{A}\ca{B}})$ such that the following is a closed, degree $0$ morphism of twisted complexes in $\barmod{B}$
	\begin{equation*}
		%\label{morphism-a-to-b}
		\begin{tikzcd}
			G_{\ca{A}} \ob{A} C_M \ob{A} M \ob{C} \dual{N}{C} \ar[r] \ar[dr, dashed, "\beta_{\ca{A}\ca{B}}"] \ar[d, "{\beta_{\ca{A}} \ob{} \mathrm{id}}"]& 0 \ar[d]\\
			F_{\ca{B}} \ob{B} N \ob{C} \dual{M}{C}[-1] \ob{A} M \ob{C} \dual{N}{C} \ar[r] & F_{\ca{B}} \ob{B} C_N
		\end{tikzcd}
	\end{equation*}

	Using that $C_N$ is an autoequivalence, that $N$ is spherical, and adjunction we deduce that $\beta$ is homotopic to $0$.

	This proves that $c_P$ is fully faithful and concludes the proof of the proposition.
\end{proof}

%% EXAMPLES
\section{Examples}
%%%
%%%
%%% SPHERICAL OBJECTS
%%%
%%%
\subsection{Spherical objects}
\label{subsect:spherical-objs}

In this subsection we consider the example of spherical objects.

Let $\ca{C}$ be a small dg-category.
Notice that every module $E \in \derived(\ca{C})$ can be considered $E$ as a $\star_{k}\text{-}\ca{C}$ bimodule where $\star_{k}$ is the dg-category with a single object such that $\Hom_{\star_k}(\star_k, \star_k) = k$; here $k$ sits in degree 0.
For such a bimodule being $\star_{k}$-perfect means that, for every $c \in \ca{C}$, the complex $E(c)$ is bounded and has finite dimensional cohomologies.

For any $E \in \derived(\ca{C})$ we will denote $\Hom^{\bullet}_{\derived(\ca{C})}(E,E)= \bigoplus_{i \in \mathbb{Z}} \Hom_{\derived(\ca{C})}(E, E[i])[-i]$ the graded algebra underlying $\text{RHom}_{\ca{C}}(E,E)$.

For the rest of this section we will assume that the category $\derived(\ca{C})^c$ has a Serre functor $\mb{S}$ which is a tensor functor.
We fix $S_{\ca{C}} \in \barbimod{C}{C}$ enhancing $\mb{S}$.

\begin{dfn}
	Let $E \in \derived(\ca{C})$, we say that $E$ is a {\it d-spherical object} if the following three conditions are satisfied:
	\begin{enumerate}
		\item $E$ is both $\star_{k}$- and $\ca{C}$-perfect.
		\item $\Hom^{\bullet}_{\derived(\ca{C})}(E,E) \simeq k[t] /t^2$, $\deg (t) = d$, as graded algebras.
		\item $E \oL{C} S_{\ca{C}} \simeq E[d]$ in $\derived(\ca{C})$. \label{Serre-action}
	\end{enumerate}
\end{dfn}

\begin{rmk}
	The reason why we require $\derived(\ca{C})^c$ to have a Serre functor is to simplify the above definition.
	If we didn't have a Serre functor, instead of \eqref{Serre-action} above we would have to require the existence of an isomorphism $E^{\widetilde{\ca{C}}} \simeq E^{\widetilde{\star_k}}[-d]$ in $\derived(\ca{C})$.
	While we could get by with this for the proof of \autoref{thm:spherical-objects}, we would run into functoriality issues in the proof of \autoref{thm:cotwist-serre-duality}
\end{rmk}

\begin{rmk}
	If $\derived(\ca{C}) \simeq \derived_{\qc}(X)$ for some smooth, projective variety $X$ of dimension $d$, then an object $E \in \derived(\ca{C})$ is d-spherical if and only if the image in $\derived_{\qc}(X)$ is spherical according to the standard definition given in \cite{Seidel-Thomas01}.
\end{rmk}

The following is well known

\begin{thm}[\cite{Seidel-Thomas01}]
	\label{thm:spherical-objects}
	For a $d$-spherical object $E$ the functor $- \stackrel{L}{\otimes}_{\star_k} E : \derived(\star_k) \rightarrow \derived(\ca{C})$ is spherical.
\end{thm}

Let us now consider d-spherical objects $E_1, \dots, E_n$.
We apply \autoref{twist-twist=twist} inductively and we obtain that the functor $t_{E_n} \circ \dots \circ t_{E_1}$ can be realised as the twist around the functor
\begin{equation}
	\label{many-spherical-objects}
	\derived(\ca{R}) \xrightarrow{- \oL{R} \left( E_n \oplus \dots \oplus E_1 \right)} \derived(\ca{C}),
\end{equation}
where the dg-category $\ca{R}$ is the dg-category with objects $\{1, \dots, n\}$ and morphisms\footnote{We think of the object $i$ as the one corresponding to the $i$-th copy of the category $\star_k$.}
\[
	\Hom_{\ca{R}}(i,j) = 
	\left\{
		\begin{array}{ll}
			0 & i < j\\
			k & i = j\\
			\mathrm{H} \overline{\mathrm{om}}_{\ca{C}}(E_i, E_j) & i > j
		\end{array}
	\right..
\]

Therefore, we see that $\derived( \ca{R} ) = \derived( R)$ where
\begin{equation}
	\label{R-spherical-objs}
	R = \bigoplus_{i=1}^{n} k \, \mathrm{id}_{E_i} \oplus \bigoplus_{i > j} \mathrm{H} \overline{\mathrm{om}}_{\ca{C}}(E_i, E_j)
\end{equation}
(considered as a subalgebra of $\mathrm{H} \overline{\mathrm{om}}_{\ca{C}} ( \oplus_{i=1}^n E_i, \oplus_{i=1}^{n} E_j)$), and \eqref{many-spherical-objects} can be rewritten as
\begin{equation}
	\label{many-spherical-objects-algebra}
	\derived(R) \xrightarrow{- \stackrel{L}{\otimes}_R \left( E_n \oplus \dots \oplus E_1 \right)} \derived(\ca{C}).
\end{equation}

\begin{rmk}
	The semiorthogonal decomposition arising from the combination of \autoref{SOD-prop-mine}, \autoref{ex:upper-triangular-dg-algebra}, and \autoref{rmk:SOD-for-all} gives us an interpretation of the category $\derived(R)$ as that of the derived category of modules over the path algebra of a quiver with relations.
	Indeed, one can think of a quiver with $n$-vertices and arrows from $i$ to $j$ labelled by $\mathrm{H} \overline{\mathrm{om}}_{\ca{C}}(E_i, E_j)$ whenever $i > j$, $0$ if $i < j$, and by $k$ if $i = j$.
	We draw the example $n = 4$
	\[
	\begin{tikzcd}[row sep = -0.2em]
		4 \ar[r] \ar[rr, bend left] \ar[rrr, bend left]& 3 \ar[r] \ar[rr, bend right] & 2 \ar[r] & 1\\
	\end{tikzcd}
	\]
\end{rmk}

\begin{ex}
	Let us give a first geometric example of the above construction; we thank Timothy Logvinenko for explaining it to us.
	Let $X$ be a smooth, projective variety, and consider two spherical objects $E,F \in \derived^{\bounded}(X)$ such that $\Hom^{\bullet}_{\derived^{\bounded}(X)}(E,F) = \Hom_{\derived^{\bounded}(X)}(E,F[1])[-1] = \mb{C}^2[-1]$.
	Consider $\mathcal{U} \in \derived^{\bounded}(\mb{P}^1 \times X)$ the universal family that parametrises non-zero extensions of $F$ by $E$ up to the action of $\mb{C}^{\times}$.
	This object has the property that its fibre over any $p \in \mb{P}^1$ gives the corresponding extension of $F$ by $E$.
	Considering $\mathcal{U}$ as a Fourier Mukai kernel, we get a functor $\Phi : \derived^{\bounded}(\mb{P}^1) \rightarrow \derived^{\bounded}(X)$ which is spherical, and whose twist is the composition $t_{E} \circ t_{F}$.
	This can be seen as an example of the above theorem considering the spherical objects $E$ and $F[1]$.
	Indeed, in this case the algebra \eqref{R-spherical-objs} is $k \oplus k^2 \oplus k$, which is the endomorphism algebra of the object $\ca{O}_{\mb{P}^1}(-1) \oplus \ca{O}_{\mb{P}^1}$ in $\derived^{\bounded}(\mb{P}^1)$, and therefore we have an equivalence
	\[
		\derived(k \oplus k^2 \oplus k) \simeq \derived_{\qc}(\mb{P}^1)
	\]
	under which the functor $f$ of \autoref{twist-twist=twist} gets identifed with $\Phi$.
\end{ex}

\begin{ex}
	Another geometric example of the above construction is given in \cite[Corollary 3.1.5]{Barb-Flop-flop}.
\end{ex}

\begin{ex}
	We thank Tobias Dyckerhoff for explaining to us the following symplectic interpretation of the above result.

	Consider $f : E \rightarrow \mb{D}$ a Lefschetz fibration with base the disk $\mb{D}$ with $n$ marked points $p_1, \dots, p_n$ corresponding to the critical points of $f$.
	Assume for simplicity that $p_i \neq 1$ for any $i$, and denote $X = f^{-1}(1)$ the smooth fibre of $f$.
	To $X$ we can associate the Fukaya--Seidel category $\mathrm{Fuk}(X)$, which in this case is generated by the vanishing cycles $S_i$'s associated to the $p_i$'s.
	The fundamental group $\pi_1 (\mb{D} \setminus \{ p_1, \dots, p_n \}, 1)$ acts on $\derived^{\pi}(\mathrm{Fuk}(X))$ via a braid group action whose generators are given by the Dehn twists around the spherical objects $S_1$, $\dots$, $S_n$.

	We can also define the directed Fukaya--Seidel category $\mathrm{Fuk}^{\rightarrow}(f)$ of $f$, see \cite[§ 6]{Seidel-vanishing-cycles}.
	This category is generated by vanishing thimbles associated to the $p_i$'s, {\it i.e.} the vanishing cycle together with the choice of a vanishing path.
	We then get a functor $\partial : \derived^{\pi}(\mathrm{Fuk}^{\rightarrow}(f)) \rightarrow \derived^{\pi}(\mathrm{Fuk}(X))$ given by sending each vanishing thimble to its boundary (which is the corresponding vanishing cycle).

	The functor $\partial$ is spherical, and the spherical twist around it is the {\it total monodromy action}.
	More precisely, $t_{\partial}$ is the composition of the Dehn twists around the $S_i$'s.
	The connection with \autoref{thm:spherical-objects} is that $\derived^{\pi}(\mathrm{Fuk}^{\rightarrow}(f))$ is the category $\derived(R)^c$ for $R$ as defined in \eqref{R-spherical-objs} with respect to the vanishing cycles $S_i$'s, see {\it ibidem}.
\end{ex}

The dg-algebra $R$ defined in \eqref{R-spherical-objs} is smooth (being the gluing of smooth dg-algebras along perfect bimodules) and proper.
Therefore, the category $\derived(R)^c$ has a Serre Duality functor given by tensor product with $R^{\ast} := \mathrm{RHom}_{k}(R,k)$, see \cite{shklyarov2007serre}.
We now describe the cotwist around \eqref{many-spherical-objects-algebra} in terms of Serre Duality for the category $\derived(R)^c$.
We have

\begin{thm}
	\label{thm:cotwist-serre-duality}
	The cotwist around \eqref{many-spherical-objects-algebra} is given by tensor product with $R^{\ast}[-1-d]$.
\end{thm}

\begin{proof}
	We give the proof for the case $n = 2$, the general case being similar.
	By \autoref{prop:description-cotwist} we know that the cotwist is described by the matrix\footnote{We can pass from $\mathrm{RHom}_{\ca{C}}(E_1,E_2)$  to the underlying graded vector space because the category of $k\text{-}k$ bimodules is semisimple.}
	\[
		\bimodule{\mathrm{Hom}_{\derived(\ca{C})}(E_1,E_1[d])[-1-d]}{0}{\mathrm{Hom}^{\bullet}_{\derived(\ca{C})}(E_1,E_2)[-1]}{\mathrm{Hom}_{\derived(\ca{C})}(E_2,E_2[d])[-1-d]}
	\]
	with structure morphisms
	\[
		\begin{aligned}
			&\mathrm{Hom}^{\bullet}_{\derived(\ca{C})}(E_1,E_2)[-1] \otimes_k \mathrm{Hom}^{\bullet}_{\ca{C}}(E_2, E_1) \xrightarrow{\mathrm{pr}_{1+d} \circ \mathrm{cmps}} \mathrm{Hom}_{\derived(\ca{C})}(E_2,E_2[d])[-1-d]\\
			&\mathrm{Hom}^{\bullet}_{\derived(\ca{C})}(E_2, E_1) \otimes_k \mathrm{Hom}^{\bullet}_{\derived(\ca{C})}(E_1,E_2)[-1] \xrightarrow{\mathrm{pr}_{1+d} \circ \mathrm{cmps}} \mathrm{Hom}_{\derived(\ca{C})}(E_1,E_1[d])[-1-d],
		\end{aligned}
	\]
	where $\mathrm{pr}_{1+d}$ is the projection to the degree $1+d$ part and $\mathrm{cmps}$ is the composition of morphisms.

	The bimodule $R^{\ast}[-1-d]$ is given by the matrix
	\[
		\bimodule{\mathrm{Hom}_{\derived(\ca{C})}(E_1,E_1)^{\ast}[-1-d]}{0}{\mathrm{Hom}^{\bullet}_{\derived(\ca{C})}(E_2,E_1)^{\ast}[-1-d]}{\mathrm{Hom}_{\derived(\ca{C})}(E_2,E_2)^{\ast}[-1-d]}
	\]
	with structure morphisms
	\[
		\begin{aligned}
			&\mathrm{Hom}^{\bullet}_{\derived(\ca{C})}(E_2,E_1)^{\ast}[-1-d] \otimes_k \mathrm{Hom}^{\bullet}_{\ca{C}}(E_2, E_1) \xrightarrow{\mathrm{ev}} \mathrm{Hom}_{\derived(\ca{C})}(E_2,E_2)^{\ast}[-1-d]\\
			&\mathrm{Hom}^{\bullet}_{\derived(\ca{C})}(E_2, E_1) \otimes_k \mathrm{Hom}^{\bullet}_{\derived(\ca{C})}(E_2,E_1)^{\ast}[-1-d] \xrightarrow{\mathrm{ev}} \mathrm{Hom}_{\derived(\ca{C})}(E_1,E_1)^{\ast}[-1-d].
		\end{aligned}
	\]

	To conclude we now consider the morphism of bimodules induced by the matrix of morphisms $\bimodule{\alpha_{11}}{0}{\alpha_{12}}{\alpha_{22}}$ where $\alpha_{ij} : \mathrm{Hom}^{\bullet}_{\derived(\ca{C})}(E_i,E_j)[-1] \rightarrow \mathrm{Hom}_{\derived(\ca{C})}^{\bullet}(E_j,E_i)^{\ast}[-1-d]$ is the isomorphism given by Serre Duality.
	This matrix of morphisms can be lifted to a morphism in the derived category of bimodules using an argument similar to \cite[Lemma 7.3]{AL-P-n-func} for the case in which the top right components of the bimodules are zero.
	As the components of the morphism are quasi-isomorphisms, the statement follows.
\end{proof}
%%%
%%%
%%% P OBJECTS
%%%
%%%
%%%
\subsection{\texorpdfstring{$\mb{P}$-objects}{P-objects}}
\label{subsection:P-objects}

In this section we consider the case of $\mb{P}$-objects.

Let $\ca{C}$ be a small dg-category over a field $k$ such that $\derived(\ca{C})^c$ has a Serre functor $\mb{S}$ which is a tensor functor: $\mb{S}(-) = - \oL{C} S_{\ca{C}}$.
$\mb{P}^n$-objects were first introduced in \cite{Huyb-Thomas06} where it was proved that given a $\mb{P}^n$-object $P$ in the derived category of a smooth projective variety $X$ one can construct an autoequivalence of $\derived^{\bounded}(X)$ called the $\mb{P}$-twist around $P$.

\begin{rmk}
	The definition of $\mb{P}^n$-objects has been later generalized to that of split $\mb{P}^n$-functors in \cite{Add-symmetries.hyperkahler}, \cite{Cautis-Flops-and-about}, and further to general $\mb{P}^n$-functors in \cite{AL-P-n-func}.
\end{rmk}

Recall the notation $\Hom_{\derived(\ca{C})}^{\bullet}(E,E) = \bigoplus_{n \in \mathbb{Z}} \text{Hom}_{\derived(\ca{C})}(E,E[n])[-n]$ for any $E \in \derived(\ca{C})$.

\begin{dfn}
	\label{def-p-obj}
	An object $P \in \derived(\ca{C})$ is said to be a $\mb{P}^n$-object if the following conditions are satisfied:
	\begin{enumerate}
		\item $P$ is both $\star_{k}$- and $\ca{C}$-perfect.
		\item $\Hom^{\bullet}_{\derived(\ca{C})}(P,P) \simeq k[t] / t^{n+1}$, $\deg(t) = 2$, as graded algebras.
		\item We have an isomorphism $P \oL{C} S_{\ca{C}} \simeq P[2n]$ in $\derived(\ca{C})$.
	\end{enumerate}
\end{dfn}

\begin{rmk}
	If $X$ is a smooth projective variety such that $\derived(\ca{C}) \simeq D_{qc}(X)$, then an object $P \in \derived(\ca{C})$ is a $\mb{P}^n$-object if and only if the corresponding object in $\derived_{\qc}(X)$ is a $\mb{P}^{n}$-object.
\end{rmk}

In \cite{Seg-autoeq-spherical-twist} the author describes two ways to realise the $\mb{P}$-twist around a $\mb{P}$-object as a spherical twist.
The first one considers $P$ as a dg-module over the dg-algebra $k[t]$, $\deg (t) = 2$, and defines the functor\footnote{To define the action of $k[t]$ on $P$ we assume that we replaced $P$ with an h-projective resolution, that we fixed a representative $\tilde{t} \in \text{Hom}_{\ca{C}}(P,P)$ for the generator of degree $2$, and we make $t$ act as $\tilde{t}$.}
\begin{equation}
	\label{P-twist-as-spherical}
	\derived( k[t] ) \xrightarrow{f_{P}(-) = - \stackrel{L}{\otimes}_{k[t]} P} \derived(\ca{C}).
\end{equation}

\begin{rmk}
	Notice that $\star_k$-perfectness implies $\star_{k[t]}$-perfectness because $k[t]$ is smooth, see e.g. \cite[pag. 7]{shklyarov2007serre}.
\end{rmk}

Then, under a technical assumption (which was subsequently shown to be always satisfied \cite{Formality-P-objects}), \cite[Proposition 4.2]{Seg-autoeq-spherical-twist} proves that \eqref{P-twist-as-spherical} is spherical, that its twist is the $\mb{P}$-twist around $P$, and that its cotwist is $[-2n-2]$.

The second construction of \cite{Seg-autoeq-spherical-twist} uses Koszul duality to rewrite \eqref{P-twist-as-spherical} in a different way.
More precisely, the object $k \in \derived(k[t])$ is compact and we have $\mathrm{RHom}_{k[t]}(k,k) \simeq k[\varepsilon] / \varepsilon^2$, with $\deg( \varepsilon) = -1$.
The $k[t]$-module $k[t,e]$, $\deg(e) = 1$, $d(e) = t$, is h-projective and gives a resolution of $k$ as a right $k[t]$-module.
Moreover, $k[t,e]$ carries a left action of $k[\varepsilon] / \varepsilon^2$ via the degree $-1$ map of $k[t]$-modules $k[t,e] \rightarrow k[t,e]$ that sends $p(t) + e q(t)$ to $q(t)$.
Hence, we get a functor
\begin{equation}
	\label{P-twist-as-spherical-2}
	\derived(k[\varepsilon] / \varepsilon^2) \xrightarrow{ - \stackrel{L}{\otimes}_{k[\varepsilon] / \varepsilon^2} P'} \derived(\ca{C})
\end{equation}
where
\[
	P' := k[t,e] \otimes_{k[t]} P = P[e] = \left\{ P[-2] \xrightarrow{t} \underset{\deg. 0}{P} \right\}.
\]
The twist around \eqref{P-twist-as-spherical-2} is the $\mb{P}$-twist around $P$, and its cotwists is given by $[-2n-2]$.

\begin{rmk}
	Construction \eqref{P-twist-as-spherical-2} was generalised to the case of split $\mb{P}^n$-functors in \cite[Theorem 5.1]{AL-P-n-func}.
\end{rmk}

\subsubsection{The \texorpdfstring{$k[t]$}{k[t]}-model}

Given $\mb{P}^n$-objects $P_1, \dots, P_m$ we can apply \autoref{twist-twist=twist} to \eqref{P-twist-as-spherical} to obtain a spherical functor
\begin{equation}
	\label{gluing-k[t]}
	\derived(R) \xrightarrow{ - \stackrel{L}{\otimes}_{R} (P_m \oplus \dots \oplus P_1)} \derived(\ca{C})
\end{equation}
whose twist is given by the composition of the $\mb{P}$-twists around $P_m, \dots, P_1$.
Here
\[
	R = \bigoplus_{i=1}^{m} k[t] \, \mathrm{id}_{P_i} \oplus \bigoplus_{i > j} \mathrm{H} \overline{\mathrm{om}}_{\ca{C}}(P_i, P_j)
\]
considered as a subalgebra of $\mathrm{H} \overline{\mathrm{om}}_{\ca{C}}(\oplus_{i=1}^m P_i, \oplus_{i=1}^m P_i)$.

The cotwist cannot be described as in \autoref{thm:cotwist-serre-duality} because $k[t]^{\ast}$ is not a shift of $k[t]$.

\begin{ex}
	A geometric example of this construction can be found in \cite[Theorem 3.3.4]{Barb-Flop-flop}.
\end{ex}

\subsubsection{The \texorpdfstring{$k[\varepsilon] / \varepsilon^2$}{k[e]}-model}

Given $\mb{P}^n$-objects $P_1, \dots, P_m$, let us denote $P'_i := k[t,e] \otimes_{k[t]} P_i = \left\{ P_i[-2] \xrightarrow{t} \underset{\mathrm{deg.} \, 0}{P_i} \right\}$.
Applying \autoref{twist-twist=twist} to \eqref{P-twist-as-spherical-2} we obtain a spherical functor
\begin{equation}
	\label{gluing-k[e]}
	\derived(R) \xrightarrow{ - \stackrel{L}{\otimes}_{R} (P'_m \oplus \dots \oplus P'_1)} \derived(\ca{C})
\end{equation}
whose twist is given by the composition of the $\mb{P}$-twists around $P_m, \dots, P_1$.
Here
\[
	R = \bigoplus_{i=1}^{m} k[\varepsilon] / \varepsilon^2 \, \mathrm{id}_{P'_i} \oplus \bigoplus_{i > j} \mathrm{H} \overline{\mathrm{om}}_{\ca{C}}(P'_i, P'_j)
\]
considered as a subalgebra of $\mathrm{H} \overline{\mathrm{om}}_{\ca{C}}(\oplus_{i=1}^m P'_i, \oplus_{i=1}^m P'_i)$.

\begin{ex}
	A geometric example of this construction can be found in \cite[Theorem 3.3.2]{Barb-Flop-flop}.
\end{ex}

\begin{rmk}
	Our hope was to prove an isomorphism $c_{\eqref{gluing-k[e]}}(-) \simeq - \stackrel{L}{\otimes}_R R^{\ast}[-2n-1]$ in analogy with \autoref{thm:cotwist-serre-duality}.
	Unfortunately, we stumble upon technical issues we do not know how to fix.
	More precisely, the components of the cotwist outside the diagonal are of the form $\mathrm{H} \overline{\mathrm{om}}_{\ca{C}}(P'_j, P'_i)[-1]$, $j < i$, and we would like to use Serre Duality to relate them to $\mathrm{H} \overline{\mathrm{om}}_{\ca{C}}(P'_i, P'_j)^{\ast}[-2n-1]$.
	The problem is that Serre Duality provides us with an isomorphism $\mathrm{Hom}^{\bullet}_{\derived(\ca{C})}(P'_j, P'_i)[-1] \simeq \mathrm{Hom}^{\bullet}_{\derived(\ca{C})}(P'_i, P'_j)[-2n-1]$ of $k[\varepsilon] / \varepsilon^2$-bimodules, but it is not clear whether one can lift this isomorphism to a quasi-isomorphism $\mathrm{H} \overline{\mathrm{om}}_{\ca{C}}(P'_j, P'_i)[-1] \rightarrow \mathrm{H} \overline{\mathrm{om}}_{\ca{C}}(P'_i, P'_j)^{\ast}[-2n-1]$ of $k[\varepsilon] / \varepsilon^2$-bimodules.
\end{rmk}

\appendix
\section{Replacing a dg-category with its additive envelope}

\label{appendix}

In this appendix, our aim is to show that to prove \autoref{twist-twist=twist} it is not restrictive to assume that the dg-categories we work with are additive.
In order to state the result precisely, let us recall the setup.

We take three small dg-categories $\ca{A}$, $\ca{B}$, $\ca{C}$ and two spherical bimodules $M \in \derived(\ca{A}\text{-}\ca{C})$, $N \in \derived(\ca{B}\text{-}\ca{C})$.
Then, we consider the bimodule $P = \rmodule{M}{N}^{t} \in \derived((\ca{B} \sqcup_{\varphi} \ca{A} )\text{-} \ca{C})$, $\varphi = M \ob{C} \dual{N}{C}$, with structure morphism $\mathrm{tr} : M \ob{C} \dual{N}{C}\ob{C} N \rightarrow M$, and we claim that this bimodule is spherical with twist and cotwist described as in \autoref{twist-twist=twist}.

We now give the following

\begin{dfn}
	\label{dfn:additive-envelope}
	Given a dg-category $\ca{D}$ we define its additive envelope as the dg-category $\ca{D}^{\mathrm{add}}$ whose objects are formal expressions
	\[
		\begin{array}{lcr}
			a_1 \oplus \dots \oplus a_n, & & a_i \in \ca{D},
		\end{array}
	\]
	and whose morphisms are given by
	\[
		\displaystyle{\Hom_{\ca{D}^{\mathrm{add}}} \left(\bigoplus_{i=1}^{n} a_i, \bigoplus_{j=1}^m b_j \right) = 
		\left(
			\begin{array}{ccc}
				\Hom_{\ca{D}}(a_1, b_1) & \dots & \Hom_{\ca{D}}(a_n, b_1)\\
				\vdots & & \vdots\\
				\Hom_{\ca{D}}(a_1, b_m) & \dots & \Hom_{\ca{D}}(a_n, b_m)
			\end{array}
		\right)},
	\]
	with degreewise graded decomposition, termwise differential, and composition given by matrix multiplication.
\end{dfn}

The dg-category $\add{\ca{D}}$ is an additive category, and we have a fully faithfull embedding of dg-categories $f_{\ca{D}} : \ca{D} \hookrightarrow \add{\ca{D}}$.
Moreover, restriction along $f_{\ca{D}}$ gives an equivalence
\begin{equation}
	\label{modules-over-additive-closure}
	\stmod{\add{\ca{D}}} \simeq \stmod{\ca{D}}.
\end{equation}
For a module $E \in \stmod{D}$ we will denote $\add{E} = E \oo{D} \add{\ca{D}}$, and similarly for bimodules.

\begin{rmk}
	The additive closure of $\ca{D}$ can be equivalently defined as the smallest additive subcategory of $\stmod{D}$ containing the image of the Yoneda embedding.
	However, describing it in abstract terms as we did above suits better our purposes.
\end{rmk}

Let us consider the categories $\add{\ca{A}}$, $\add{\ca{B}}$ and the bimodules $\add{M}$, $\add{N}$, and $\add{\varphi}$.
We set $\ca{R} = \add{\ca{B}} \times_{\add{\varphi}[1]} \add{\ca{A}}$ and $\tilde{P} = \rmodule{\add{M}}{\add{N}}^t \in \barbimod{R}{C}$ the bimodule given by the structure morphism
\[
	\begin{aligned}
		\add{\rho} : & \, \add{\ca{A}} \oo{A} M \ob{C} \dual{N}{C} \oo{B} \add{\ca{B}} \ob{\add{\ca{B}}} \add{N}\xrightarrow{\simeq} \add{\ca{A}} \oo{A} M \ob{C} \dual{(\add{N})}{C} \ob{\add{\ca{B}}} \add{N} \rightarrow\\
		& \, \xrightarrow{\mathrm{id}^{\otimes 2} \ob{} \mathrm{tr}} \add{\ca{A}} \oo{A} M = \add{M}.
	\end{aligned}
\]

We aim to show that we can translate all the questions we have regarding \eqref{functor-f-P} in terms of $\tilde{P}$.
In order to do so, we prove the following lemmas.

\begin{lem}
	\label{lem:twist-cotwist-swapped}
	Let $\ca{D}$, $\ca{E}$, $\ca{F}$ three small dg-categories over a field $k$.
	If $f : \derived(\ca{E}) \rightarrow \derived(\ca{F})$ is tensor functor and $\Phi: \derived(\ca{D}) \rightarrow \derived(\ca{E})$ is a tensor equivalence, then $f$ is spherical if and only if $f \Phi$ is, and we have
	\begin{equation}
		\label{relation-twist-cotwist}
		\begin{array}{lcr}
			t_f \simeq t_{f \Phi} &
			\text{and} &
			c_{f \Phi} \simeq \Phi^{-1} \circ c_f \circ \Phi.
		\end{array}
	\end{equation}
\end{lem}

\begin{proof}
	Notice that $f$ has a left (resp. right) adjoint that is a tensor functor if and only if $f\Phi$ does.
	Therefore, by \autoref{thm:spherical-functors}, to conclude it is enough to prove the description of the twist and the cotwist.
	In order to do so, we prove that the unit and counit morphism of $f \Phi \dashv \Phi^R f^R$ are given as follows
	\[
		\begin{array}{lcr}
			f  \Phi  \Phi^R  f^R \xrightarrow{f \circ \mathrm{tr}_{\Phi} \circ f^{R}} f  f^R \xrightarrow{\mathrm{tr}_f} \mathrm{id}
		& \mathrm{and} & 
			\mathrm{id} \xrightarrow{\mathrm{act}_{\Phi}} \Phi^R \Phi \xrightarrow{\Phi^R \circ \mathrm{act}_{f} \circ \Phi} \Phi^R f^R f \Phi.
		\end{array}
	\]

	We prove the description of the counit, the proof for the unit being similar.
	The map $\Phi \Phi^R f^R \rightarrow f^R$ corresponding to the identity is given by $\mathrm{tr}_{\Phi} \circ f^R$.
	Then, the counit map is given as
	\[
		f  \Phi  \Phi^R  f \xrightarrow{f \circ \mathrm{tr}_{\Phi} \circ f^{R}} f  f^R \xrightarrow{\mathrm{tr}_f} \mathrm{id},
	\]
	which proves the claim.
\end{proof}

\begin{prop}
	\label{prop:swap-categories}
	Let $\ca{D}$, $\ca{E}$, $\ca{F}$ three small dg-categories over a field $k$ and let $F : \ca{D} \rightarrow \ca{E}$ be a dg-functor.
	Fix $M \in \derived(\ca{D}\text{-}\ca{F})$ and $N  \in \derived(\ca{E}\text{-}\ca{F})$, and set $M' = \ind{F^{\mathrm{opp}}}(M)$, $N' = \res{F^{\mathrm{opp}}}(N)$.
	
	If $\ind{F}$ is an equivalence, then $M$ (resp. $N$) is spherical if and only if $M'$ (resp. $N'$) is.
\end{prop}

\begin{proof}
	Notice that the functor induced by $M'$ is $\res{F}(-) \oL{\ca{D}} M \colon \derived(\ca{E}) \rightarrow \derived(\ca{F})$, and the one induced by $N'$ is $\ind{F}(-) \oL{\ca{E}} N \colon \derived(\ca{D}) \rightarrow \derived(\ca{F})$.

	To conclude we notice that $\ind{F}^{-1} = \res{F}$ by \autoref{prop:KL-ind-res} and therefore we can apply \autoref{lem:twist-cotwist-swapped}.
\end{proof}

Notice that we have a fully faithful functor
	\begin{equation}
		\label{functor-to-add}
			\ca{B} \sqcup_{\varphi} \ca{A} \hookrightarrow \add{\ca{B}} \sqcup_{\add{\varphi}} \add{\ca{A}} \hookrightarrow \add{\ca{B}} \times_{\add{\varphi}[1]} \add{\ca{A}}
	\end{equation}
which induces an equivalence of derived categories by the same arguments as in \autoref{subsub:relations}.
Applying \autoref{prop:swap-categories} to \eqref{functor-to-add} and using relations \eqref{relation-twist-cotwist}, we get

\begin{prop}
	\label{prop:additive-not-restrictive}
	The bimodule $\tilde{P}$ is spherical if and only if $P$ is too.
	Moreover, $t_{\tilde{P}} \simeq t_{P}$ and the cotwist around $c_P$ has the structure claimed in \autoref{twist-twist=twist} if and only if the cotwist around $\tilde{P}$ has the same structure with $M$ and $N$ replaced by $\add{M}$ and $\add{N}$, respectively.
\end{prop}

\bibliographystyle{alphaurl}

\begin{thebibliography}{{Tab}07}

\bibitem[Add16]{Add-symmetries.hyperkahler}
Nicolas Addington.
\newblock New derived symmetries of some hyperk\"{a}hler varieties.
\newblock {\em Algebr. Geom.}, 3(2):223--260, 2016.
\newblock \href {https://doi.org/10.14231/AG-2016-011}
  {\path{doi:10.14231/AG-2016-011}}.

\bibitem[AL17]{Anno-Log-17}
Rina Anno and Timothy Logvinenko.
\newblock Spherical {DG}-functors.
\newblock {\em J. Eur. Math. Soc. (JEMS)}, 19(9):2577--2656, 2017.
\newblock \href {https://doi.org/10.4171/JEMS/724}
  {\path{doi:10.4171/JEMS/724}}.

\bibitem[AL19]{AL-P-n-func}
Rina {Anno} and Timothy {Logvinenko}.
\newblock {$\mathbb{P}^n$-functors}.
\newblock {\em arXiv e-prints}, page arXiv:1905.05740, May 2019.
\newblock \href {http://arxiv.org/abs/1905.05740} {\path{arXiv:1905.05740}}.

\bibitem[AL20]{Anno-Logvinenko-Bar-Categories}
Rina Anno and Timothy Logvinenko.
\newblock {Bar Category of Modules and Homotopy Adjunction for Tensor
  Functors}.
\newblock {\em International Mathematics Research Notices}, 06 2020.
\newblock rnaa066.
\newblock \href {https://doi.org/10.1093/imrn/rnaa066}
  {\path{doi:10.1093/imrn/rnaa066}}.

\bibitem[{Bar}20]{Barb-Flop-flop}
Federico {Barbacovi}.
\newblock {Spherical functors and the flop-flop autoequivalence}.
\newblock {\em arXiv e-prints}, page arXiv:2007.14415, July 2020.
\newblock \href {http://arxiv.org/abs/2007.14415} {\path{arXiv:2007.14415}}.

\bibitem[BK89]{BonKap89}
A.~I. Bondal and M.~M. Kapranov.
\newblock Representable functors, {S}erre functors, and reconstructions.
\newblock {\em Izv. Akad. Nauk SSSR Ser. Mat.}, 53(6):1183--1205, 1337, 1989.
\newblock \href {https://doi.org/10.1070/IM1990v035n03ABEH000716}
  {\path{doi:10.1070/IM1990v035n03ABEH000716}}.

\bibitem[Cau12]{Cautis-Flops-and-about}
Sabin Cautis.
\newblock Flops and about: a guide.
\newblock In {\em Derived categories in algebraic geometry}, EMS Ser. Congr.
  Rep., pages 61--101. Eur. Math. Soc., Z\"{u}rich, 2012.

\bibitem[HK19]{Formality-P-objects}
Andreas Hochenegger and Andreas Krug.
\newblock Formality of {$\mathbb{P}$}-objects.
\newblock {\em Compos. Math.}, 155(5):973--994, 2019.
\newblock \href {https://doi.org/10.1112/s0010437x19007218}
  {\path{doi:10.1112/s0010437x19007218}}.

\bibitem[HLS16]{Halpern-Shipman16}
Daniel Halpern-Leistner and Ian Shipman.
\newblock Autoequivalences of derived categories via geometric invariant
  theory.
\newblock {\em Adv. Math.}, 303:1264--1299, 2016.
\newblock \href {https://doi.org/10.1016/j.aim.2016.06.017}
  {\path{doi:10.1016/j.aim.2016.06.017}}.

\bibitem[HT06]{Huyb-Thomas06}
Daniel Huybrechts and Richard Thomas.
\newblock {$\mathbb{P}$}-objects and autoequivalences of derived categories.
\newblock {\em Math. Res. Lett.}, 13(1):87--98, 2006.
\newblock \href {https://doi.org/10.4310/MRL.2006.v13.n1.a7}
  {\path{doi:10.4310/MRL.2006.v13.n1.a7}}.

\bibitem[Kel94]{Kel94}
Bernhard Keller.
\newblock Deriving {DG} categories.
\newblock {\em Ann. Sci. \'{E}cole Norm. Sup. (4)}, 27(1):63--102, 1994.
\newblock URL: \url{http://www.numdam.org/item?id=ASENS_1994_4_27_1_63_0}.

\bibitem[KL15]{KL15}
Alexander Kuznetsov and Valery~A. Lunts.
\newblock Categorical resolutions of irrational singularities.
\newblock {\em Int. Math. Res. Not. IMRN}, (13):4536--4625, 2015.
\newblock \href {https://doi.org/10.1093/imrn/rnu072}
  {\path{doi:10.1093/imrn/rnu072}}.

\bibitem[Lun10]{Lunts-Cat-res-sing}
Valery~A. Lunts.
\newblock Categorical resolution of singularities.
\newblock {\em J. Algebra}, 323(10):2977--3003, 2010.
\newblock \href {https://doi.org/10.1016/j.jalgebra.2009.12.023}
  {\path{doi:10.1016/j.jalgebra.2009.12.023}}.

\bibitem[Nee96]{Neeman-Brown-Rep}
Amnon Neeman.
\newblock The {G}rothendieck duality theorem via {B}ousfield's techniques and
  {B}rown representability.
\newblock {\em J. Amer. Math. Soc.}, 9(1):205--236, 1996.
\newblock \href {https://doi.org/10.1090/S0894-0347-96-00174-9}
  {\path{doi:10.1090/S0894-0347-96-00174-9}}.

\bibitem[{Nee}18]{Neeman-Grothendieck-duality}
Amnon {Neeman}.
\newblock {Grothendieck duality made simple}.
\newblock {\em arXiv e-prints}, page arXiv:1806.03293, Jun 2018.
\newblock \href {http://arxiv.org/abs/1806.03293} {\path{arXiv:1806.03293}}.

\bibitem[Seg18]{Seg-autoeq-spherical-twist}
Ed~Segal.
\newblock All autoequivalences are spherical twists.
\newblock {\em Int. Math. Res. Not. IMRN}, (10):3137--3154, 2018.
\newblock \href {https://doi.org/10.1093/imrn/rnw326}
  {\path{doi:10.1093/imrn/rnw326}}.

\bibitem[Sei01]{Seidel-vanishing-cycles}
Paul Seidel.
\newblock Vanishing cycles and mutation.
\newblock In {\em European {C}ongress of {M}athematics, {V}ol. {II}
  ({B}arcelona, 2000)}, volume 202 of {\em Progr. Math.}, pages 65--85.
  Birkh\"{a}user, Basel, 2001.

\bibitem[Shk07]{shklyarov2007serre}
D.~Shklyarov.
\newblock On serre duality for compact homologically smooth dg algebras.
\newblock 2007.
\newblock \href {http://arxiv.org/abs/math/0702590}
  {\path{arXiv:math/0702590}}.

\bibitem[ST01]{Seidel-Thomas01}
Paul Seidel and Richard Thomas.
\newblock Braid group actions on derived categories of coherent sheaves.
\newblock {\em Duke Math. J.}, 108(1):37--108, 2001.
\newblock \href {https://doi.org/10.1215/S0012-7094-01-10812-0}
  {\path{doi:10.1215/S0012-7094-01-10812-0}}.

\bibitem[Tab05]{Tabuada-dgCAT}
Goncalo Tabuada.
\newblock Une structure de cat\'{e}gorie de mod\`eles de {Q}uillen sur la
  cat\'{e}gorie des dg-cat\'{e}gories.
\newblock {\em C. R. Math. Acad. Sci. Paris}, 340(1):15--19, 2005.
\newblock \href {https://doi.org/10.1016/j.crma.2004.11.007}
  {\path{doi:10.1016/j.crma.2004.11.007}}.

\bibitem[{Tab}07]{Tabuada-07}
Goncalo {Tabuada}.
\newblock {Theorie homotopique des DG-categories}.
\newblock {\em arXiv e-prints}, page arXiv:0710.4303, Oct 2007.
\newblock \href {http://arxiv.org/abs/0710.4303} {\path{arXiv:0710.4303}}.

\end{thebibliography}

\end{document}